\documentclass[12pt]{amsart}
\usepackage{epsf}
\usepackage{psfrag}
\usepackage{fullpage}
\usepackage{float}
\usepackage{mathrsfs}
\usepackage{amsfonts}
\usepackage[centertags]{amsmath}
\usepackage{amssymb}
\usepackage{amsthm}
\usepackage{graphicx}
\usepackage{float}
\usepackage[all]{xy}
\usepackage{tikz-cd}
\usepackage{tikz}
\usetikzlibrary[shapes]
\usepackage{multirow}
\usepackage{caption}
\usepackage{xparse} 
\usepackage{comment}
\usepackage{lscape}
\usepackage{dsfont}
\usepackage{inputenc}
\usepackage{subcaption} 
\usepackage{enumerate,enumitem}
\usepackage[margin=1in]{geometry} 
\usetikzlibrary{matrix,arrows,decorations.pathmorphing}
\usepackage{listing} 
\usepackage{hyperref}
\hypersetup{
bookmarksnumbered,
pdfstartview={FitH},
breaklinks=true,
linkcolor=blue,
urlcolor=blue,
citecolor=blue,
bookmarksdepth=2
}
\usepackage{cleveref}
\usepackage{ytableau}
\usepackage{adjustbox}
\usepackage{booktabs} 
\usepackage{mathtools} 
\usepackage{xcolor}
\usepackage{stmaryrd}
\usepackage{leftindex}
\usepackage{mathscinet}
\usepackage{bm}

\theoremstyle{plain}
   \newtheorem{theorem}{Theorem}[section]
   \newtheorem{proposition}[theorem]{Proposition}
   \newtheorem{lemma}[theorem]{Lemma}
   \newtheorem{corollary}[theorem]{Corollary}
   \newtheorem{conjecture}[theorem]{Conjecture}
   
\theoremstyle{definition}
   \newtheorem{definition}[theorem]{Definition}
   
   \newtheorem{example}[theorem]{Example}

   \newtheorem{remark}[theorem]{Remark}

   \newtheorem{introthm}{Theorem}

\numberwithin{equation}{section}

\newcommand{\CC}{{\mathbb {C}}}

\newcommand{\RR}{{\mathbb {R}}}
\newcommand{\ZZ}{{\mathbb {Z}}}
\newcommand{\NN}{{\mathbb {N}}}

\newcommand{\g}{\mathfrak{g}}
\newcommand{\hatg}{\hat{\mathfrak{g}}}
\newcommand{\Uqghat}{\mathcal{U}_q(\hat{\mathfrak{g}})}
\newcommand{\Uqghatdual}{\mathcal{U}_q(\leftindex^L {\hat{\mathfrak{g}}})}
\newcommand{\Uqghatsigma}{\mathcal{U}_q(\hat{\mathfrak{g}}^{\sigma})}

\newcommand{\Uqhaffine}{\mathcal{U}_q\tilde{\mathfrak{h}}}

\newcommand{\bpsi}{\mathbf{\Psi}}

\newcommand{\ringYa}{\mathbb{Z}[Y_{i,a}^{\pm 1}]_{i \in I,a \in \mathbb{C}^*}}

\DeclareMathOperator*{\im}{Im}

\DeclareMathOperator{\Hom}{Hom}

\DeclareMathOperator{\sgn}{sgn}

\title{Langlands branching rule for type B snake modules}
\author{Jingmin Guo \and Jian-Rong Li \and Keyu Wang}
\address{Faculty of Mathematics, University of Vienna, Oskar Morgenstern Platz 1, 1090 Vienna, Austria}
\email{jingmin.guo@univie.ac.at, jianrong.li@univie.ac.at, keyu.wang@univie.ac.at}

\begin{document}

\begin{abstract}
We prove that each snake module of the quantum Kac--Moody algebra of type $B_n^{(1)}$ admits a Langlands dual representation, as conjectured by Frenkel and Hernandez (\textit{Lett. Math. Phys.} (2011) 96:217-261). Furthermore, we establish an explicit formula, called the Langlands branching rule, which gives the multiplicities in the decomposition of the character of a snake module of the quantum Kac--Moody algebra of type $B_n^{(1)}$ into a sum of characters of irreducible representations of its Langlands dual algebra.
\end{abstract}

\maketitle
\tableofcontents

\section{Introduction}
\subsection{Background}
In the Langlands correspondence, two Lie algebras are called dual to each other if they are defined by dual Cartan data. In the case of simple Lie algebras, this means that their associated Cartan matrices are transposes of one another.

For example, in the case of finite-dimensional simple Lie algebras, the Langlands dual of $\mathfrak{so}_{2n+1}$ (type $B_n$) is $\mathfrak{sp}_{2n}$ (type $C_n$).

Our goal is to establish a relationship between the representation theories of Langlands dual algebras. Let us describe the problem, starting with the finite-dimensional case.

Let $P$ and $\leftindex^L{P}$ be the weight lattices of $\mathfrak{so}_{2n+1}$ and $\mathfrak{sp}_{2n}$ respectively. Consider the linear map $\Pi_{\mathbb{Q}} : P \otimes_{\mathbb{Z}} \mathbb{Q} \to \leftindex^L{P} \otimes_{\mathbb{Z}} \mathbb{Q}$ which sends the $i$-th fundamental weight $\omega_i \in P$ to the $i$-th fundamental weight $\check{\omega}_i \in \leftindex^L{P}$ if $i \neq n$, and sends $\omega_n$ to $\frac{1}{2}\check{\omega}_n$. The map $\Pi_{\mathbb{Q}}$ is equivariant under the Weyl group action.

Let $\chi$ (resp. $\leftindex^L{\chi}$) be the character map of $\mathfrak{so}_{2n+1}$ (resp. $\mathfrak{sp}_{2n}$) which takes values in $\mathbb{Z}[P]$ (resp. $\mathbb{Z}[\leftindex^L{P}]$). Then given any representation $V$ of $\mathfrak{so}_{2n+1}$, the sum of those terms in $\Pi_{\mathbb{Q}}(\chi(V))$ that lie in $\mathbb{Z}[\leftindex^L{P}]$ is $W$-invariant (we denote this partial sum by $\Pi(\chi(V))$), and hence lies in the image $\im(\leftindex^L{\chi})$. This suggests that the sum $\Pi(\chi(V))$ can be expressed as the character of a virtual representation of $\mathfrak{sp}_{2n}$, or equivalently,
\begin{equation}\label{introeq1}
    \Pi(\chi(V)) = \sum_W m_{V,W} \leftindex^L{\chi}(W),
\end{equation}
where the sum is over a finite set of irreducible representations $W$ of $\mathfrak{sp}_{2n}$, and $m_{V,W} \in \mathbb{Z}$.

In the work of Frenkel and Hernandez \cite{frenkel2011langlandsreps}, it was conjectured that the coefficients $m_{V,W}$ are all non-negative. Their approach is based on quantum groups. They verified their conjecture in some cases. This conjecture was solved by McGerty \cite{mcgerty2010langlands}, where the coefficients $m_{V,W}$ are computed explicitly. Such a formula as in \eqref{introeq1} is called the Langlands branching rule of finite type.

We remark that although the problem concerns only classical Lie algebras, the proof in \cite{mcgerty2010langlands} is also based on quantum groups.

\subsection{Frenkel--Hernandez conjecture}
We are interested in the quantum affine analogue of this problem. Namely, consider a \textit{finite-dimensional} representation $V$ of the quantum affine algebra $\mathcal{U}_q(\widehat{\mathfrak{so}}_{2n+1})$ (type $B_n^{(1)}$) and its character $\chi(V)$ which lies in $\mathbb{Z}[P]$, where $P$ is the weight lattice of $\mathfrak{so}_{2n+1}$. Here $\chi$ is known as the $\mathfrak{g}$-characters of finite-dimensional representations of quantum affine algebras $\Uqghat$.

It is natural to ask whether $\Pi(\chi(V))$ can also be decomposed into a \textit{positive} sum of characters of representations of the dual algebra, and if so, how to compute the coefficients. Note that we are considering finite-dimensional representations, rather than highest weight representations which are infinite-dimensional for affine types.

The Langlands dual of an affine algebra $\hatg$ is denoted by $\leftindex^L{\hatg}$. Recall that when $\hatg = \widehat{\mathfrak{so}}_{2n+1}$, its Langlands dual is the twisted affine Lie algebra $\leftindex^L{\hatg} = {\widehat{\mathfrak{sl}}_{2n}}^{\sigma}$ (type $A_{2n-1}^{(2)}$), see \cite[Section~4]{kac1990infinite}. The character for twisted types is denoted by $\chi^{\sigma}$, to distinguish it from non-twisted types. Note that when $\leftindex^L{\hatg} = {\widehat{\mathfrak{sl}}_{2n}}^{\sigma}$, its character $\chi^{\sigma}$ takes values in $\mathbb{Z}[\leftindex^L{P}]$, where $\leftindex^L{P}$ is the weight lattice of type $C_n$.

Two conjectures were proposed by Frenkel and Hernandez in \cite{frenkel2011langlandsfinite} concerning the Langlands duality in affine types.

The first conjecture (\cite[Conjecture~2.2]{frenkel2011langlandsfinite}) states that given an irreducible representation $V$ of $\Uqghat$, there exists an irreducible representation of $\Uqghatdual$, denoted by ${}^L{V}$ and called a Langlands dual representation of $V$, which satisfies
\begin{itemize}
    \item Let $\lambda \in P$ be the highest weight of $V$. Then ${}^L{V}$ has highest weight $\Pi(\lambda)$.
    \item $\Pi(\chi(V)) - \chi^{\sigma}({}^L{V}) \in \mathbb{N}[\leftindex^L{P}]$.
\end{itemize}
This conjecture was verified when $V$ is a Kirillov-Reshetikhin (KR) module of $\Uqghat$, and in this case, ${}^L{V}$ is a KR module of $\Uqghatdual$ \cite[Theorem~3.11]{frenkel2011langlandsfinite}.

The second conjecture (\cite[Conjecture~2.4]{frenkel2011langlandsfinite}) is much stronger, asserting that for any irreducible representation $V$, there is a decomposition as in \eqref{introeq1} with all coefficients non-negative. In particular, provided the first conjecture is true, the Langlands dual representation ${}^L{V}$ will appear as a term in \eqref{introeq1} with coefficient $m_{V,{}^L{V}} =1$. However, the positivity conjecture remains widely open; for example, it is unknown even for KR modules.

\subsection{Main results}
This paper addresses the positivity problem and the branching rule for the Langlands dual pair of types $B_n^{(1)}$ and $A_{2n-1}^{(2)}$. We establish the Langlands branching rule from type $B_n^{(1)}$ to type $A_{2n-1}^{(2)}$ for a large class of representations known as snake modules, which include all minimal affinizations \cite{chari1995minimalnonsimp,chari1996minimalirr,chari1996minimalsimp}, and thus include KR modules as a special case. Snake modules are also well known in the representation theory of $p$-adic groups. Via the quantum affine Schur--Weyl duality \cite{chari1996quantum}, they correspond to ladders in Zelevinsky's classification \cite{zelevinsky1980induced,lapid2014on}. Moreover, snake modules play an important role in the categorification of cluster algebras through representations of quantum Kac--Moody algebras \cite{duan2019cluster}.

We first verify that the first conjecture of Frenkel--Hernandez holds for snake modules.
\begin{introthm}[= Theorem~\ref{thm: langlands dual rep A to B}]\label{introthm1}
    Let $V$ be a snake module of type $B_n^{(1)}$. Then $V$ has a Langlands dual representation ${}^L{V}$, which is a snake module of type $A_{2n-1}^{(2)}$.
\end{introthm}
This provides additional evidence for \cite[Conjecture~2.2]{frenkel2011langlandsfinite}, extending from KR modules to snake modules.

Next, we treat the positivity conjecture of Frenkel--Hernandez. More precisely, we provide an explicit formula for the branching rule of snake modules of $\Uqghat$.

\begin{introthm}[= Theorem~\ref{thm: langlands branching rule}]\label{introthm2}
    Let $V$ be a snake module of type $B_n^{(1)}$. Then we have
    \begin{equation}\label{introeq2}
        \Pi(\chi(V)) = \sum_W \chi^{\sigma}(W),
    \end{equation}
    where the sum is over snake modules $W$ of type $A_{2n-1}^{(2)}$ satisfying a betweenness condition, and all coefficients are $1$. In particular, the Langlands dual representation ${}^L{V}$ constructed in Theorem~\ref{introthm1} appears in the decomposition.
\end{introthm}
This result is an affine analogue of the Langlands branching rule established by McGerty \cite{mcgerty2010langlands}, and we therefore also refer to our formula as the Langlands branching rule.

Moreover, we will see that the condition for $W$ to appear in \eqref{introeq2} is given by a betweenness condition, which takes a form analogous to the well-known betweenness condition in the branching rule when restricting $\mathfrak{gl}_n$ modules to $\mathfrak{gl}_{n-1}$. This condition can be briefly described as follows: snake modules are labeled by two strictly increasing sequences of integers $(l_1,\dots,l_T)$ and $(r_1,\dots,r_T)$ satisfying $r_t - l_t \ge 0$ for all $1 \le t \le T$. If $V$ is such a snake module, then the representations $W$ appearing in the decomposition are precisely those snake modules labeled by two sequences $(l'_1, \dots, l'_T)$ and $(r'_1,\dots,r'_T)$ such that for all $1 \le t \le T$, $l'_t$ are fixed by $l'_t = l_t$ and $r'_t$ are between two integers $r_{t-1} < r'_t \le r_t$, where we set $r_0 = -\infty$.

Let us remark that in contrast to the restriction branching rule, which naturally appears when restricting a module to a subalgebra, the Langlands branching rule is proved only at the level of characters, and an algebraic explanation for the appearance of such a branching rule remains to be discovered.

The approach taken in this paper differs from that in \cite{frenkel2011langlandsfinite}. Our method consists of three parts:

\begin{enumerate}
    \item We prove that the characters of snake modules of non-twisted type $A_{2n-1}^{(1)}$ and those of twisted type $A_{2n-1}^{(2)}$ are related by a folding map. The folding process is proposed as a conjecture in \cite{hernandez2006kirillov} and proved for KR modules therein. We first show that the folding process applies to snake modules (Proposition~\ref{prop: folding snake formula}).

    \item There is a combinatorial description of $q$-characters $\chi_q(V)$ of snake modules of types $B_n^{(1)}$ and $A_n^{(1)}$ developed by Mukhin and Young, known as the path description \cite{mukhin2012path}. The path description allows us to relate characters of type $B_n^{(1)}$ and non-twisted type $A_{2n-1}^{(1)}$ (Proposition~\ref{prop: bijection type B and pair of type A}, Theorem~\ref{thm: bijection PB and S}).

    \item Combining these two relations, the proof of Theorem~\ref{introthm2} is reduced to establishing an identity of characters of snake modules of type $A_{n-1}^{(1)}$ (Theorem~\ref{thm: type A equality}). We prove this identity in Section~\ref{sec:equality}, using the determinant formula \cite{bittmann2023on, brito2024alternating, chenevier2008characters, lapid2014on, lapid2018geometric, tadic1995on}. 
\end{enumerate}

We remark that an isomorphism between the Grothendieck rings of certain categories of representations of quantum Kac--Moody algebras of Langlands dual types was established in \cite{kashiwara2019categorical}. However, the duality constructed in this paper differs from the one in loc. cit.

\subsection{Organization}
This paper is organized as follows. In Section~\ref{sec:preliminaries}, we review the necessary background, including quantum Kac--Moody algebras and their representation theory, the folding process, and the Langlands duality. In Section~\ref{sec:path description}, we recall the Mukhin--Young path description of snake modules of types $A_{2n-1}^{(1)}$ and $B_n^{(1)}$. In Section~\ref{sec:Langlands B to A}, we define a map that associates a path of type $A_{2n-1}$ to a path of type $B_n$. As a consequence, we prove Theorem~\ref{introthm1}. Section~\ref{sec:equality} is a preparation for Section~\ref{sec:Langlands branching rule}. In Section~\ref{sec:equality}, we prove an identity between the characters of snake modules of type $A_{n-1}^{(1)}$ (Theorem~\ref{thm: type A equality}). In Section~\ref{sec:Langlands branching rule}, we use this result to establish the Langlands branching rule for snake modules. Finally, in Section~\ref{sec:A to B}, we discuss the duality in the opposite direction, from type $A_{2n-1}^{(2)}$ to $B_n^{(1)}$.

\vspace{5mm}

\paragraph{\textbf{Acknowledgments}}
We are grateful to Matheus Brito and Vyjayanthi Chari for valuable discussions.
This work was partially supported by the Austrian Science Fund (FWF): P-34602, Grant DOI: 10.55776/P34602, and PAT 9039323, Grant DOI 10.55776/PAT9039323.

\section{Preliminaries}\label{sec:preliminaries}
In Section~\ref{subsec: prem QAA}, we recall the definition of quantum Kac--Moody algebras, which include both non-twisted and twisted quantum affine algebras. In Section~\ref{subsec: prem non twisted} and Section~\ref{subsec: prem twisted type}, we review the representation theory of non-twisted quantum affine algebras and twisted quantum affine algebras, respectively. In Section~\ref{subsec: folding}, we recall known results that establish the relations between twisted and non-twisted types via the folding process. Finally, in Section~\ref{subsec: prem Langlands}, we recall the definition of Langlands dual algebras and the conjectures of Frenkel and Hernandez on Langlands dual representations.

\subsection{Quantum Kac--Moody algebras}\label{subsec: prem QAA}

For $N \in \NN^*$, an $N \times N$ matrix $C$ is called a generalized Cartan matrix if $C_{ij} \in \mathbb{Z}$ ($\forall i,j$), $C_{ii}=2$ ($\forall i$), $C_{ij} \leq 0$ ($\forall i \neq j$) and $C_{ij}=0$ if and only if $C_{ji}=0$. We suppose that $C$ is symmetrizable, that is, there exists a diagonal matrix $D = \mathrm{diag}(d_1,\dots,d_N)$ such that the matrix $DC$ is symmetric. A generalized Cartan matrix $C$ is said to be of affine type if its determinant is $0$ and all of its principal minors are positive. In the affine type, the positive integers $d_i$ are unique up to a scalar. We fix them so that $\mathrm{min}(d_i) = 1$. See \cite[Section~4.5]{kac1990infinite} for the classification of affine Cartan matrices. 

Let $q \in \mathbb{C}^*$ be such that $q$ is not a root of unity. Denote by $q_i = q^{d_i}$.

\begin{definition}
    Let $C$ be a generalized Cartan matrix of affine type. The quantum Kac--Moody algebra $\mathcal{U}_q(\mathfrak{g}(C))$ is the associative $\mathbb{C}$-algebra with generators $k_{i}^{\pm 1} , e_i^{\pm} \; (1 \leq i \leq N)$, and relations
	\begin{equation*}
		\begin{split}
			&k_i k_i^{-1} = k_i^{-1}k_i = 1, \; k_ik_j = k_jk_i,\\
			&k_ie_j^{\pm} = q_i^{\pm C_{i,j}}e_j^{\pm}k_i,\\
			&[e_i^{+},e_j^{-}] = \delta_{i,j} \frac{k_i-k_i^{-1}}{q_i - q_i^{-1}}, \forall i,j,\\
			&\sum_{r=0}^{1-C_{i,j}} (-1)^r(e_i^{\pm})^{(1-C_{i,j}-r)} e_j^{\pm} (e_i^{\pm})^{(r)} = 0 \; \text{for} \; i \neq j,
		\end{split}
	\end{equation*}
	where $(e_i^{\pm})^{(r)} = \frac{(e_i^{\pm})^r}{[r]_{q_i}!}$, $[r]_{q_i}! = [r]_{q_i} [r-1]_{q_i} \cdots [1]_{q_i}$, and $ [m]_{q_i}$ is the $q$-number $[m]_{q_i} = \frac{q_i^m-q_i^{-m}}{q_i-q_i^{-1}}$. We use the convention that $[0]_q = 0$, $[0]_q! = 1$.
\end{definition}

In this paper, we are interested in the affine types $A_{2n-1}^{(1)}$, $B_{n}^{(1)}$ and $A_{2n-1}^{(2)}$, $n \geq 2$. As in \cite[Section~4.5]{kac1990infinite}, we use the set of indices as follows:

In type $A_{2n-1}^{(1)}$, $N = 2n$. The set of indices is $\hat{I} = I \sqcup \{0\}$, consisting of the index set $I = \{1,2,\dots,2n-1\}$ of the Lie algebra $\mathfrak{sl}_{2n}$ and an extra index $0$. The generalized Cartan matrix of type $A_{2n-1}^{(1)}$ is 
\[
\left(
\scalebox{0.8}{$\begin{array}{cccccccc}
2 & -1 &        &        &        &        &        &     -1   \\
-1 & 2 & -1     &        &        &        &        &        \\
   & -1 & 2 & -1     &        &        &        &        \\
   &    & -1 & \ddots & \ddots &        &        &        \\
   &    &    & \ddots & \ddots & -1     &        &        \\
   &    &    &        & -1 & 2 & -1     &        \\
   &    &    &        &    & -1 & 2 & -1 \\
 -1  &    &    &        &    &    & -1 & 2
\end{array}$}
\right),
\]
where the order of labels of the matrix is $0,1,2,\dots,2n-1$.

In type $B_n^{(1)}$, $N = n+1$. The set of indices is $\hat{I} = I \sqcup \{0\}$, consisting of the index set $I = \{1,2,\dots,n\}$ of the Lie algebra $\mathfrak{so}_{2n+1}$ and an extra index $0$. Here $n$ is the index for the short root of $\mathfrak{so}_{2n+1}$. The generalized Cartan matrix of type $B_{n}^{(1)}$ is 
\[
\left(
\scalebox{0.8}{$\begin{array}{cccccccc}
2 & 0 &  -1      &        &        &        &        &        \\
0 & 2 & -1     &        &        &        &       &        \\
-1   & -1 & 2 & -1     &        &        &        &        \\
   &    & -1 & \ddots & \ddots &        &        &        \\
   &    &    & \ddots & \ddots & -1     &        &        \\
   &    &    &        & -1 & 2 & -1     &        \\
   &    &    &        &    & -1 & 2 & -1 \\
   &    &    &        &    &    & -2 & 2
\end{array}$}
\right),
\]
where the order of labels of the matrix is $0,1,2,\dots,n$.

In type $A_{2n-1}^{(2)}$, $N = n+1$. The set of indices is $\hat{I}^{\sigma} = I \sqcup \{0\}$, consisting of the index set $I = \{\bar{1},\dots,\bar{n}\}$ and an extra index $0$. Here we remark that the set $I$ arises from the orbits of the involution on the Dynkin diagram of type $A_{2n-1}$. We use the notation with a bar to emphasize that we are in the twisted case. The generalized Cartan matrix of type $A_{2n-1}^{(2)}$ is 
\[
\left(
\scalebox{0.8}{$\begin{array}{cccccccc}
2 & 0 &  -1      &        &        &        &        &        \\
0 & 2 & -1     &        &        &        &       &        \\
-1   & -1 & 2 & -1     &        &        &        &        \\
   &    & -1 & \ddots & \ddots &        &        &        \\
   &    &    & \ddots & \ddots & -1     &        &        \\
   &    &    &        & -1 & 2 & -1     &        \\
   &    &    &        &    & -1 & 2 & -2 \\
   &    &    &        &    &    & -1 & 2
\end{array}$}
\right),
\]
where the order of labels of the matrix is $0,\bar{1},\bar{2},\dots,\bar{n}$.

Quantum Kac--Moody algebras can be defined as affinizations of finite-type quantum groups via Drinfeld’s realization \cite{drinfeld1987new, chari1998twisted, damiani2015from}. We do not list the generators and relations here, as they will not be used in an essential way. We note, however, that there exists a large commutative subalgebra $\Uqhaffine$, generated by elements $\phi_{i,\pm k}^{\pm}$ ($i \in I, k \geq 0$) in Drinfeld's realization.

Denote by
    \[\phi_{i}^{\pm}(u) = \sum_{k=0}^{\infty} \phi_{i,\pm k}^{\pm} u^{\pm k}, \quad i \in I.\]

In accordance with common terminology, we will also call a quantum Kac--Moody algebra of non-twisted affine type as a non-twisted quantum affine algebra, and a quantum Kac--Moody algebra of twisted affine type as a twisted quantum affine algebra.


\subsection{Non-twisted type}\label{subsec: prem non twisted}
Let $\mathfrak{h}$ be the Cartan subalgebra of the finite-dimensional Lie algebra $\g$, 
and $\mathfrak{h}^*$ be its dual vector space.
Let $\omega_i \in \mathfrak{h}^*$ ($i \in I$) be the fundamental weights of $\g$, and $\alpha_i \in \mathfrak{h}^*$ ($i \in I$) be the simple roots of $\g$. $\mathfrak{h}^*$ is equipped with an inner product $\langle, \rangle$ such that $\langle \alpha_i,\omega_j \rangle = d_i \delta_{ij}$.

Let $\Uqghat$ be a quantum Kac--Moody algebra of non-twisted affine type. In this paper, it will be of type $A_{2n-1}^{(1)}$ or $B_n^{(1)}$. For a finite-dimensional representation $V$ of $\Uqghat$, an element $\lambda \in \mathfrak{h}^*$ is called a weight of $V$ if the space
\[V_{\lambda} := \{v \in V~|~k_i v = q^{\langle \alpha_i,\lambda \rangle}v, \forall i \in I \}\]
is non-zero. 

Denote by $P = \bigoplus_{i \in I} \ZZ \omega_i$ the weight lattice of $\g$.

The \textit{usual character} of a finite-dimensional representation $V$ of $\Uqghat$ is defined as
\[\chi(V) = \sum_{\lambda \in P} \dim(V_{\lambda})[{\lambda}].\]
This is an element in the abelian group $\ZZ[P]$ generated by formal symbols $[\lambda]$, $\lambda \in P$. We remark that the usual character is defined as a $\g$-character for the finite-dimensional simple Lie algebra $\g$, it should not be confused with the $\hatg$-character of the affine Kac--Moody algebra as considered in \cite{kac1990infinite,lusztig1993introduction}. However, since we consider only finite-dimensional representations, the $\g$-character contains the same information as the $\hatg$-character in this context.

Then we recall the $q$-characters of finite-dimensional representations of $\Uqghat$.

For any $\Uqghat$-module $V$, a set of complex numbers $\bpsi = (\psi_{i,\pm k}^{\pm})_{i \in I, k \in \NN}$ is called an $l$-weight of $V$ if the space 
$\{v \in V~|~\phi_{i,\pm k}^{\pm}.v = \psi_{i,\pm k}^{\pm}v \}$ is non-zero. It was proved in \cite{frenkel1999q} that any $l$-weight of a finite-dimensional representation $V$ of $\Uqghat$ is of the form
\[\sum_{k \geq 0} \psi^{\pm}_{i, \pm k} u^{\pm k} = q^{\deg(P_i) - \deg(R_i)}\frac{P_i(uq_i^{-1})R_i(uq_i)}{P_i(uq_i)R_i(uq_i^{-1})} \]
expanded in $\CC \llbracket u \rrbracket$ (resp. in $\CC \llbracket u^{-1} \rrbracket$), for some polynomials $P_i$ and $R_i$ satisfying $P_i(0) = R_i(0) = 1$.

The $q$-characters of finite-dimensional representations of $\Uqghat$ were defined by Frenkel and Reshetikhin (and their Yangian analogues were defined by Knight) \cite{frenkel1999q,knight1995spectra}. Let $V$ be a finite-dimensional representation of $\Uqghat$,

\[\chi_q(V) = \sum_{\bpsi} \dim(V_{\bpsi}) m_{\bpsi} \in \ringYa,\]
where $m_{\bpsi} = \prod_{i\in I,a\in \CC^*} Y_{i,a}^{p_{i,a}-r_{i,a}}$, $P_i(u) = \prod_{a \in \CC^*} (1-au)^{p_{i,a}}$, $R_i(u) = \prod_{a \in \CC^*} (1-au)^{r_{i,a}}$,
and
	\[V_{\bpsi} := \{v \in V~|~\exists p \geq 0, \forall i \in I,k \geq 0, (\phi_{i,\pm k}^{\pm} - \psi_{i,\pm k}^{\pm})^p.v = 0 \}.\]
    
The irreducible finite-dimensional representations of $\Uqghat$ are classified by $l$-highest weight representations \cite{chari1991quantum,chari1995quantum}. An $l$-highest weight representation of $\Uqghat$ is a $\Uqghat$-module $V$ generated by a vector $v_0 \in V$ such that 
	\[x_{i,l}^+.v_0 = 0, \; \phi_{i,\pm k}^{\pm}.v_0 = \psi_{i,\pm k}^{\pm}v_0, \quad \forall i \in I, l \in \ZZ, k \in \NN, \]
for some complex numbers $\psi_{i,\pm k}^{\pm} \in \CC$ ($i \in I, k \in \NN$). We call such $\bpsi = (\psi_{i,\pm k}^{\pm})_{i \in I, k \in \NN}$ and its corresponding monomial $m_{\bpsi} \in \ringYa$ a highest $l$-weight of $V$.

A monomial $m = \prod_{i \in I,a \in \CC^*} Y_{i,a}^{u_{i,a}} \in \ringYa$ is called dominant if $u_{i,a} \geq 0$, $\forall i \in I, a \in \CC^*$.

\begin{theorem}[\cite{chari1991quantum,chari1995quantum}]
    For each dominant monomial $m = \prod_{i \in I,a \in \CC^*} Y_{i,a}^{u_{i,a}} \in \ringYa$, there is a unique irreducible finite-dimensional $l$-highest weight representations of $\Uqghat$ with highest $l$-weight $m$. We denote this irreducible representation by $L(m)$.

    Moreover, any irreducible finite-dimensional representation of $\Uqghat$ is of this form.
\end{theorem}

The usual character $\chi(V)$ can be easily recovered from the $q$-character $\chi_q(V)$ as follows. The map $\prod_{i,a} Y_{i,a}^{u_{i,a}} \mapsto [\sum_{i,a} u_{i,a} \omega_i]$ generates a homomorphism $\ZZ[Y_{i,a}^{\pm 1}]_{i \in I,a \in \CC^*} \to \ZZ[P]$. The image of $\chi_q(V)$ under this homomorphism is exactly the usual character $\chi(V)$ \cite[Theorem~3]{frenkel1999q}.

\subsection{Twisted type}\label{subsec: prem twisted type}
The character theory for quantum Kac--Moody algebras of twisted affine type is parallel to that of non-twisted type \cite{chari1998twisted,hernandez2010kirillov,wang2023qq}. Let $\Uqghatsigma$ be a quantum Kac--Moody algebra of twisted affine type. 

In this paper, $\Uqghatsigma$ will be of type $A_{2n-1}^{(2)}$. Recall that in this case, the index set $I = \{\bar{1},\dots,\bar{n}\}$. Let $\leftindex^L {\mathfrak{h}}^* = \bigoplus_{i \in I} \CC \check{\omega}_i$ and  $\leftindex^L {P} = \bigoplus_{i \in I} \ZZ \check{\omega}_i$. For now, the symbols $\check{\omega}_i$, $\leftindex^L {\mathfrak{h}}^*$ and $\leftindex^L {P}$ are simply notations used to distinguish the twisted types from the non-twisted types; the reason for this choice of notation will become clear in Section~\ref{subsec: prem Langlands}.

For a finite-dimensional representation $V$ of $\Uqghatsigma$, an element $\lambda \in \leftindex^L{\mathfrak{h}}^*$ is called a weight of $V$ if the space 
\[V_{\lambda} := \{v \in V~|~k_i v = q^{\langle \check{\alpha}_i,\lambda \rangle}v, \forall i \in I \}\]
is non-zero. Here $\langle, \rangle$ is the inner product on $\leftindex^L {\mathfrak{h}}^*$ such that $\langle \check{\alpha}_i,\check{\omega}_j \rangle = \check{d}_i \delta_{ij}$, where $\check{d}_i$ are entries of the diagonal matrix $D$ which symmetrize the generalized Cartan matrix $C$ of type $A_{2n-1}^{(2)}$.

The usual character of a finite-dimensional representation $V$ of $\Uqghatsigma$ is defined as
\[\chi^{\sigma}(V) = \sum_{\lambda \in \leftindex^L{P}} \dim(V_{\lambda})[{\lambda}],\]
which is an element in the abelian group $\ZZ[\leftindex^L{P}]$ generated by formal symbols $[\lambda]$, $\lambda \in \leftindex^L {P}$.

In Drinfeld realization of twisted quantum affine algebras of type $A_{2n-1}^{(2)}$, there are elements 
$$\phi_{i,\pm k}^{\pm},\quad i \in I, k \in \NN,$$ 
where $\phi_{\bar{n},\pm k}^{\pm} = 0$ if $k$ is odd.

Similarly, for any $\Uqghatsigma$-module $V$, a set of complex numbers $\bpsi = (\psi_{i,\pm k}^{\pm})_{i \in I, k \in \NN}$, is called an $l$-weight of $V$ if the space 
$\{v \in V~|~\phi_{i,\pm k}^{\pm}.v = \psi_{i,\pm k}^{\pm}v \}$ is non-zero, and any $l$-weight of a finite-dimensional representation $V$ of $\Uqghat$ is of the form
\[\sum_{k \geq 0} \psi^{\pm}_{i, \pm k} u^{\pm k} = q^{\deg(P_i) - \deg(R_i)}\frac{P_i(uq_i^{-1})R_i(uq_i)}{P_i(uq_i)R_i(uq_i^{-1})} \]
expanded in $\CC \llbracket u \rrbracket$ (resp. in $\CC \llbracket u^{-1} \rrbracket$), for some polynomials $P_i$ and $R_i$ which satisfy that $P_i(0) = R_i(0) = 1$ and that $P_{\bar{n}}$ and $R_{\bar{n}}$ have only even degree terms.

As in \cite{wang2023qq}, we consider the commutative ring
\[\mathcal{Z} := \ZZ[Z_{i,a}^{\pm 1}]_{i \in I,a \in \CC^*}/(Z_{\bar{n},a}^{\pm 1} = Z_{\bar{n},-a}^{\pm 1})_{a \in \CC^*}.\]

The $q$-character of a finite-dimensional representation $V$ of $\Uqghatsigma$ is

\[\chi_q^{\sigma}(V) = \sum_{\bpsi} \dim(V_{\bpsi}) m_{\bpsi} \in \mathcal{Z},\]
where $m_{\bpsi} = \prod_{i\in I,a\in \CC^*} Z_{i,a}^{p_{i,a}-r_{i,a}}$, $P_i(u) = \prod_{a \in \CC^*} (1-au)^{p_{i,a}}$, $R_i(u) = \prod_{a \in \CC^*} (1-au)^{r_{i,a}}$ when $i \neq \bar{n}$, and $P_{\bar{n}}(u) = \prod_{a \in \CC^*} (1-a^2u^2)^{p_{\bar{n},a}}$, $R_{\bar{n}}(u) = \prod_{a \in \CC^*} (1-a^2u^2)^{r_{\bar{n},a}}$,
and the $l$-weight space
	\[V_{\bpsi} := \{v \in V~|~\exists p \geq 0, \forall i \in I,k \geq 0, (\phi_{i,\pm k}^{\pm} - \psi_{i,\pm k}^{\pm})^p.v = 0 \}.\]

Similar to non-twisted type, irreducible finite-dimensional representations of $\Uqghatsigma$ are classified by $l$-highest weight representations $L(m)$ with monomial $m = \prod_{i \in I,a \in \CC^*} Z_{i,a}^{u_{i,a}} \in \mathcal{Z}$ such that $u_{i,a} \geq 0$, $\forall i \in I, a \in \CC^*$, which are defined in the same way as in non-twisted type.

As in the non-twisted case, the usual character of a finite-dimensional representation $V$ of $\Uqghatsigma$ can be recovered from its $q$-character as follows.

The map $\prod_{i,a} Z_{i,a}^{u_{i,a}} \mapsto [\sum_{i,a} u_{i,a}\check{\omega}_i]$ generates a homomorphism $\mathcal{Z} \to \ZZ[\leftindex^L{P}]$. The image of the $q$-character $\chi_q^{\sigma}(V)$ under this homomorphism coincides with the usual character $\chi^{\sigma}(V)$.

\subsection{Folding characters}\label{subsec: folding}
Let $\g$ be a finite-dimensional simply-laced simple Lie algebra. In this section we assume that $\g$ is of type $A_{2n-1}$. Let $\Uqghat$ be its associated non-twisted quantum affine algebra of type $A_{2n-1}^{(1)}$ and $\Uqghatsigma$ be the twisted quantum affine algebra of type $A_{2n-1}^{(2)}$. Even though no relation is known between the algebra structures of $\Uqghat$ and $\Uqghatsigma$, the character theories of their representations are closely related \cite{hernandez2010kirillov,wang2023qq}.

It was proved \cite[Theorem~4.15]{hernandez2010kirillov} that the ring homomorphism
\begin{equation}\label{eqn: folding pi}
    \begin{split}
    \pi: \ZZ[Y^{\pm 1}_{i,a}]_{1 \le i \le 2n-1,a\in \CC^*} \to \mathcal{Z},\\
    \begin{cases}
    Y_{i,a}^{\pm 1} \mapsto Z_{\bar{i},a}^{\pm 1},&~\text{if}~i \le n,\\
    Y_{i,a}^{\pm 1} \mapsto Z_{\overline{2n-i},-a}^{\pm 1},&~\text{if}~i > n,
    \end{cases}
\end{split}
\end{equation}
induces a ring homomorphism between the Grothendieck ring of finite-dimensional representations
\begin{equation}\label{eqn:bar pi of Grothendieck rings}
\bar{\pi}: \mathrm{Rep}(\Uqghat) \to \mathrm{Rep}(\Uqghatsigma).
\end{equation}
Here $\mathrm{Rep}(\Uqghat)$ (resp. $\mathrm{Rep}(\Uqghatsigma)$) is the Grothendieck ring of the category of finite-dimensional representations of $\Uqghat$ (resp. of $\Uqghatsigma$).

If we restrict to the subcategory of representations whose $l$-weights are monomials in $\ZZ[Y_{i,q^n}^{\pm 1}]_{i\in I,n \in\ZZ}$, then it is conjectured that $\bar{\pi}$ maps the classes  of irreducible representations to classes of irreducible representations.

\begin{conjecture}[{\cite[Section~4.4]{hernandez2010kirillov},\cite[Conjecture~2.20]{wang2023qq}}]\label{conj: folding}
    Let $m$ be a dominant monomial in $\ZZ[Y_{i,q^n}^{\pm}]_{i\in I,n \in\ZZ}$, let $\pi(m)$ be the corresponding dominant monomial in $\mathcal{Z}$ under the map $\pi$. Then 
    \[\bar{\pi} ([L(m)]) = [L(\pi(m))].\]
\end{conjecture}
This conjecture has been proven for Kirillov-Reshetikhin modules \cite[Theorem~4.15]{hernandez2010kirillov}. In Proposition~\ref{pro:snake modules conjecture}, We will show that it also holds for snake modules of type $A_{2n-1}$.

\subsection{Langlands duality}\label{subsec: prem Langlands}
We begin with an explanation of the symbols $\leftindex^L {\mathfrak{h}}^*$, $\leftindex^L {P}$, $\check{\omega}_i$ and $\check{\alpha}_i$ used in Section~\ref{subsec: prem twisted type}.

Let $\Uqghat$ be the quantum affine algebra of type $B_n^{(1)}$. Recall that the index set $I = \{1,\dots,n\}$. Recall that $\mathfrak{h}^* = \bigoplus_{i \in I} \CC \omega_i$, and $P = \bigoplus_{i \in I} \ZZ \omega_i$ is the weight lattice of the simple Lie algebra of type $B_n$. Let $R = \bigoplus_{i \in I}\ZZ \alpha_i \subset P$ be the root lattice of type $B_n$.

Consider the dual space $\leftindex^L {\mathfrak{h}}^* = \Hom(\mathfrak{h}^*,\CC)$ and the dual lattice to the root lattice $R$, denoted by $\leftindex^L{P} = \bigoplus_{i \in I} \ZZ \check{\omega}_{\bar{i}} \subset \leftindex^L {\mathfrak{h}}^*$ with dual basis $\check{\omega}_{\bar{i}} (\alpha_j) = \delta_{ij}$. Let $\leftindex^L {R} = \bigoplus_{i \in I} \ZZ \check{\alpha}_{\bar{i}} \subset \leftindex^L {P}$ be the dual lattice of the weight lattice $P$ with dual basis $\check{\alpha}_{\bar{i}} (\omega_j) = \delta_{ij}$.

Then the inner product $\langle,\rangle$ on $\mathfrak{h}^*$ induces an inner product $\langle,\rangle$ on $\leftindex^L {\mathfrak{h}}^*$. One calculates that $\langle \check{\alpha}_{\bar{i}},\check{\omega}_{\bar{j}} \rangle = \check{d}_{\bar{i}} \delta_{ij}$, where $\check{d}_{\bar{i}}$ are exactly the entries of the diagonal matrix $D$ which symmetrize generalized Cartan matrix of type $A_{2n-1}^{(2)}$.

This phenomenon arises from the fact that the generalized Cartan matrices of type $B_{n}^{(1)}$ and of type $A_{2n-1}^{(2)}$ are transposes of each other. Two quantum Kac--Moody algebras are said to be \textit{Langlands dual} if their associated generalized Cartan matrices are transposes of one another. In particular, the non-twisted quantum affine algebra $\Uqghat$ of type $B_n^{(1)}$ and the twisted quantum affine algebra of type $A_{2n-1}^{(2)}$ are Langlands dual. For this reason, we denote the latter by $\Uqghatdual$.

Let $P'$ be the sublattice of $P$ defined by 
$$P' = \bigoplus_{i \in I} \mathbb{Z} \check{d}_{\bar{i}} \omega_i,$$ where $\check{d}_{\bar{i}} = 1$ for $i \neq n$ and $\check{d}_{\bar{n}} = 2$. 

Following \cite{frenkel2011langlandsfinite,frenkel2011langlandsreps}, there is a bijective linear map
\begin{equation}\label{eqn: Langlands character dual map}
    \Pi : P' \to \leftindex^L{P}, \quad \check{d}_{\bar{i}} \omega_i \mapsto \check{\omega}_{\bar{i}}, \forall i.
\end{equation}

\begin{remark}\label{remark: P prime}
    Notice that the root lattice $R$ is a sublattice of $P'$. Therefore, for any irreducible representation $V$ of $\Uqghat$, if there exists a non-zero weight space $V_{\lambda} \neq \{0\}$ with $\lambda \in P'$, then all weights of $V$ lie in $P'$. In particular, it coincides with the notation $\Pi(\chi(V))$ used in the introduction. For this reason, we will focus on finite-dimensional irreducible representations whose highest weight lies in $P'$.
\end{remark}

Representations of quantum Kac--Moody algebras of Langlands dual types have been found to exhibit intriguing relations \cite{frenkel2011langlandsfinite,frenkel2011langlandsreps}. More precisely, it is conjectured that each irreducible representation of $\Uqghat$ admits a Langlands dual representation, in the following sense.

\begin{definition}
    Given an irreducible finite-dimensional representation $V$ of $\Uqghat$ of highest weight $\lambda \in P'$, an irreducible representation of $\Uqghatdual$ is called a Langlands dual representation of $V$, denoted by ${}^L{V}$, if it has highest weight $\Pi(\lambda)$ and 
    \[\chi^{\sigma}({}^L{V}) \preceq \Pi (\chi(V)),\]
    where the inequality means $\Pi (\chi(V)) - \chi^{\sigma}({}^L{V}) \in \NN [\leftindex^L {P}]$, i.e. the multiplicity of each weight on the left-hand side is less than or equal to that on the right-hand side.
\end{definition}

\begin{conjecture}[{\cite[Conjecture~2.2]{frenkel2011langlandsfinite}}]
    For any irreducible finite-dimensional representation $V$ of $\Uqghat$ whose highest weight lies in $P'$, there exists an irreducible representation ${}^L{V}$ of $\Uqghatdual$, which is a Langlands dual representation of $V$.
\end{conjecture}

The conjecture is stated for all pairs of Langlands dual algebras, beyond just the duality between $B_n^{(1)}$ and $A_{2n-1}^{(2)}$. It has been verified in the case where $V$ is a Kirillov-Reshetikhin module \cite[Theorem~2.3]{frenkel2011langlandsfinite}. Moreover, an algorithm was developed to compute the $q$-character of ${}^L {V}$, known as the interpolating $(q,t)$-characters.

Furthermore, it is conjectured that the expression $\Pi (\chi(V))$ is the character of an actual representation of $\Uqghatdual$.

\begin{conjecture}[{\cite[Conjecture~2.4]{frenkel2011langlandsfinite}}]
    For any irreducible finite-dimensional representation $V$ of $\Uqghat$ whose highest weight lies in $P'$, there is a representation of $\Uqghatdual$, denoted by $W$, such that
    \[\chi^{\sigma}(W) = \Pi(\chi(V)).\]

    Equivalently, $\Pi(\chi(V))$ can be decomposed as
    \[\Pi(\chi(V)) = \sum_{i} c_i \chi^{\sigma}(W_i),\]
    where $W_i$ are irreducible representations of $\Uqghatdual$ and $c_i > 0$.
\end{conjecture}

In this paper, we prove that these two conjectures hold when $V$ is a snake module of type $B_n^{(1)}$. Moreover, in this case, we provide an explicit formula for the terms $W_i$. For this reason, we refer to our formula as the Langlands branching rule.

\section{Path description}\label{sec:path description}
In this section, we recall the path descriptions introduced by Mukhin and Young \cite{mukhin2012path}. Path descriptions are combinatorial methods of describing the $q$-characters of snake modules of quantum affine algebras in type $A_{2n-1}^{(1)}$ or $B_n^{(1)}$, based on the Frenkel--Mukhin algorithm \cite{frenkel2001combinatorics}. Throughout this paper, the number $n \geq 2$ will be fixed. 

\subsection{Paths}\label{subsec: paths}
In this subsection, we review fundamental concepts of the paths defined in \cite{mukhin2012path}. We will adopt slightly different notation for convenience. For the original definitions and additional details, readers can refer to \cite[Section~5]{mukhin2012path}.

Let $I = \{1,2,\dots,2n-1\}$ (resp. $I = \{1,2,\dots,n\}$) be the set of indices of Dynkin diagram of type $A_{2n-1}$ (resp. type $B_n$). In type $B_n$, the label $n \in I$ corresponds to the short simple root.

Define 
\[\mathcal{X}^A = \{(i,k) \in I \times \ZZ~|~k \equiv n+i+1 \pmod 2\},\] and 
\[\mathcal{X}^B = \{(i,k) \in I \times \ZZ~|~k \equiv 2n+2i+2 \pmod 4\}.\]

\begin{definition}\label{def: paths}
    The sets of paths are defined as follows.
    \begin{itemize}
        \item Type $A_{2n-1}$: for $(i,k) \in \mathcal{X}^A$,
        \begin{align*}
           \mathscr{P}_{i,k}^{A}:=\{\big( (0,y_0), &(1,y_1), \dots, (2n, y_{2n}) \big)~| \\ &y_0= i+k, y_{2n}=2n-i+k, 
            \text{and $\lvert y_{r+1}-y_r \rvert=1$ for $0\leq r \leq 2n-1$}\}.
        \end{align*}

        \item Type $B_n$: fix an $\epsilon \in \RR$, $0 < \epsilon < 1/2$. For $(i,k) \in \mathcal{X}^B$,
        \begin{align*}
        &&\mathscr{P}_{i,k}^{B}:=\{\big( (0,y_0), &(2,y_1), \dots, (2n-4, y_{n-2}), (2n-2,y_{n-1}), (2n-1,y_n), \\
           && &(2n-1,z_n), (2n,z_{n-1}),\dots,(4n-4,z_1),(4n-2,z_0) \big)~|~y_n > z_n, \text{and} \\
           && & y_0= 2i+k,  \lvert y_{n}-y_{n-1} \rvert=1+\epsilon,  \text{and $\lvert y_{r+1}-y_r \rvert=2$ for $0\leq r\leq n-2$,}\\
           &&z_0= 4n-&2i+k-2,  \lvert z_{n}-z_{n-1} \rvert=1+\epsilon,  \text{and $\lvert z_{r+1}-z_r \rvert=2$ for $0\leq r\leq n-2$}\}.
        \end{align*}
    \end{itemize}
\end{definition}

Paths of type $B_n$ will be illustrated in Figure~\ref{illustrate of P_ik}.

For simplicity, we use the notation $\mathscr{P}_{i,k}$ to represent either $\mathscr{P}_{i,k}^A$ or $\mathscr{P}_{i,k}^B$ when the type is irrelevant or can be inferred from context. A similar omission of upper indices will apply to other symbols as well, and we will not repeat this explanation.

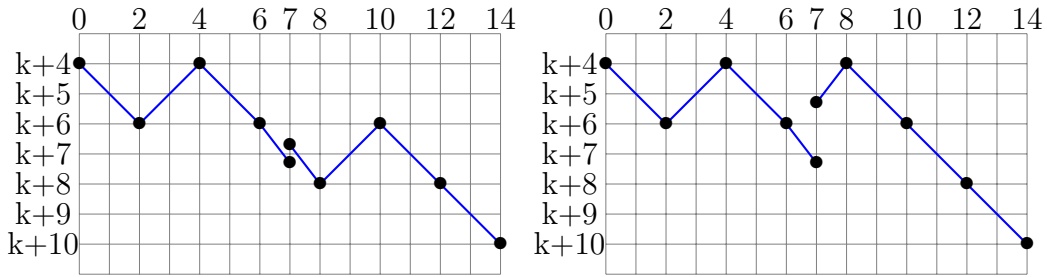
\begin{figure}[htp]
    \centering
\begin{tikzpicture}[scale=0.4]
\draw[step=1cm, gray, very thin] (0, 0) grid (14, 8);
\node at (0, 8.5) {0}; \node at (1, 8.5) {};\node at (2, 8.5) {2};
\node at (3, 8.5) {};\node at (4, 8.5) {4}; \node at (5, 8.5) {};\node at (6, 8.5) {6};
\node at (7, 8.5) {7}; \node at (8, 8.5) {8};\node at (9, 8.5) {};\node at (10, 8.5) {10}; 
\node at (11, 8.5) {}; \node at (12, 8.5) {12}; \node at (13, 8.5) {}; \node at (14, 8.5) {14}; 
\node at (-1.2, 0) { }; \node at (-1.2, 1) {k+10};
\node at (-1.2, 2) {k+9};\node at (-1.2, 3) {k+8};\node at (-1.2, 4) {k+7}; \node at (-1.2, 5) {k+6}; \node at (-1.2, 6) {k+5}; \node at (-1.2, 7) {k+4}; \node at (-1.2, 8) { }; 
\draw[blue, thick] (0,7)--(2,5)--(4,7) --(6,5)--(7,3.7);
\draw[blue, thick] (7,4.3)--(8,3)--(10,5) --(14,1);
\node at (0,7) {$\bullet$};\node at (2,5) {$\bullet$};\node at (4,7) {$\bullet$};\node at (6,5) {$\bullet$};\node at (7,3.7) {$\bullet$};\node at (7,4.3) {$\bullet$};\node at (8,3) {$\bullet$};\node at (10,5) {$\bullet$};\node at (14,1) {$\bullet$};\node at (12,3) {$\bullet$};
\begin{scope}[xshift=17.5cm]
\draw[step=1cm, gray, very thin] (0, 0) grid (14, 8);
\node at (0, 8.5) {0}; \node at (1, 8.5) {};\node at (2, 8.5) {2};
\node at (3, 8.5) {};\node at (4, 8.5) {4}; \node at (5, 8.5) {};\node at (6, 8.5) {6};
\node at (7, 8.5) {7}; \node at (8, 8.5) {8};\node at (9, 8.5) {};\node at (10, 8.5) {10}; 
\node at (11, 8.5) {}; \node at (12, 8.5) {12}; \node at (13, 8.5) {}; \node at (14, 8.5) {14}; 

\node at (-1.2, 0) { }; \node at (-1.2, 1) {k+10};
\node at (-1.2, 2) {k+9};\node at (-1.2, 3) {k+8};\node at (-1.2, 4) {k+7}; \node at (-1.2, 5) {k+6}; \node at (-1.2, 6) {k+5}; \node at (-1.2, 7) {k+4}; \node at (-1.2, 8) { }; 

\draw[blue, thick] (0,7)--(2,5)--(4,7) --(6,5)--(7,3.7);
\draw[blue, thick] (7,5.7)--(8,7)--(10,5) --(14,1);

\node at (0,7) {$\bullet$};\node at (2,5) {$\bullet$};\node at (4,7) {$\bullet$};\node at (6,5) {$\bullet$};\node at (7,3.7) {$\bullet$};\node at (7,5.7) {$\bullet$};\node at (8,7) {$\bullet$};\node at (10,5) {$\bullet$};\node at (14,1) {$\bullet$};\node at (12,3) {$\bullet$};
\end{scope}

\end{tikzpicture}
 \caption{ Two paths of type $B_4$ in $\mathscr{P}^B_{2,k}$. Two consecutive points are connected by a blue edge. Note that the vertical coordinates increase from top to bottom.}
    \label{illustrate of P_ik} 
\end{figure}

\begin{definition}
    \begin{enumerate}
        \item A \textit{path} $p$ is an element of the set $\mathscr{P}_{i,k}$. It is a finite sequence of points in $\RR^2$.

        \item Denote by $(x,y)\in p$ if the point $(x,y)$ is an element of the finite sequence $p$.

        \item In type $B_n$, we call the subsequence 
        \[\big( (0,y_0), (2,y_1), \dots, (2n-4, y_{n-2}), (2n-2,y_{n-1}), (2n-1,y_n)\big)\] the left branch of the path, and
        \[\big( (2n-1,z_n), (2n,z_{n-1}),(2n+2,z_{n-2}), \dots,(4n-4,z_1),(4n-2,z_0) \big)\] the right branch.
    \end{enumerate}
\end{definition}

\begin{remark}\label{remark: path new definition}
    Our definitions correspond to the original definitions in the following way:

    \begin{itemize}
        \item In $\mathscr{P}_{i,k}^{A}$, we changed the condition $k \equiv i+1 \pmod 2$ in \cite[Section~5.1]{mukhin2012path} to $k \equiv n+i+1 \pmod 2$.

        \item A path in $\mathscr{P}_{n,k}^{B}$ is a pair of paths in $(\mathscr{P}_{n,k-1},\mathscr{P}_{n,k+1})$ of type $B_n$ in \cite[Section~5.1]{mukhin2012path}, because we will always use two paths of type $B_n$ in pair.
    \end{itemize}
\end{remark}

We recall the definitions of upper corners and lower corners, following \cite[Section~5.2]{mukhin2012path}.

\begin{definition}
     Denote by $X^B = \big( \{0,1,\dots,n-1\} \times 2\ZZ \big) \sqcup \big(\{n\} \times (2\ZZ+1)\big)$.
    Define an injective map
        \begin{equation}\label{eqn: tau}
        \begin{split}
        \tau: X^B &\rightarrow \ZZ \times \ZZ  \\
        (j,\ell)&\mapsto 
        \begin{cases}
        (2j,\ell), &\text{if}~j<n~\text{and}~\ell \equiv 2n+2j+2 \pmod 4,\\
        (4n-2-2j,\ell), &\text{if}~j<n~\text{and}~\ell \equiv 2n+2j\pmod 4,\\
        (2n-1,\ell), &\text{if}~j=n.
        \end{cases}
        \end{split}
        \end{equation}
\end{definition}

\begin{itemize}
    \item 
In type $A_{2n-1}$: let $p=\big( (j_r,y_r)\big)_{0\leq r\leq 2n}\in\mathscr{P}_{i,k}^{A}$. Define
\[
C_{p}^{-}:=\{(j_r,y_r)~|~1 \leq r \leq 2n-1, y_{r-1}=y_{r}-1=y_{r+1}\},
\]
called the set of \textit{lower corners}, and
\[
C_{p}^{+}:=\{(j_r,y_r)~|~1 \leq r \leq 2n-1, y_{r-1}=y_{r}+1=y_{r+1}\},
\]
called the set of \textit{upper corners}.

\item
In type $B_n$, let $p=\big( (j_r,y_r)\big)_{0\leq r\leq 2n+1}\in\mathscr{P}_{i,k}^{B}$. Define
\begin{align*}
C_{p}^{-}:=\tau^{-1}\{(j_r,y_r)\in p~|~j_r\notin \{0,2n-1,4n-2\}, y_{r-1}<y_{r}, y_{r+1}<y_r \} \\ \sqcup \{(n,\ell)~|~(2n-1, \ell+\epsilon)\in p~\text{and}~(2n-1, \ell-\epsilon)\notin p \}, 
\end{align*}
called the set of \textit{lower corners}, and
\begin{align*}
C_{p}^{+}:=\tau^{-1}\{(j_r,y_r)\in p~|~j_r\notin \{0,2n-1,4n-2\}, y_{r-1}>y_{r}, y_{r+1}>y_r \} \\
\sqcup \{(n,\ell)~|~(2n-1, \ell-\epsilon)\in p~\text{and}~(2n-1, \ell+\epsilon)\notin p \},
\end{align*}
called the set of \textit{upper corners}.
\end{itemize}

\subsection{Snake positions}\label{subsection:snakeposition}

In this subsection, we recall the definition of snakes and make a slight modification for the purposes of this paper.

We introduce the notion of shortened snake positions. Let $\mathcal{X}$ be the set $\mathcal{X}^A$ or $\mathcal{X}^B$ defined in Section~\ref{subsec: paths}.
\begin{definition}\label{def: shortened snakes}
    For $(i,k)\in \mathcal{X}$, another point $(i',k') \in \mathcal{X}$ is said in \textit{shortened snake position} with respect to $(i,k)$ if 
    \begin{itemize}
        \item 
            Type $A_{2n-1}$:\[ k'-k\geq \lvert i'-i\rvert +2.\]
        \item 
            Type $B_n$:
            \[k'-k\geq 2\lvert i'-i\rvert +4.\]
    \end{itemize}
\end{definition}

Shortened snakes are defined as follows.
\begin{definition}\label{def: definition of snakes}
    A \textit{shortened snake} is a finite sequence $(i_t, k_t)_{1 \leq t \leq T}$, $T\in \mathbb{N}^*$, of points in $\mathcal{X}$ such that for $2 \leq t \leq T$, $(i_t,k_t)$ is in shortened snake position with respect to $(i_{t-1},k_{t-1})$. 
\end{definition}

The shortened snakes are related to the notion of snakes in \cite[Section~4.2]{mukhin2012path} and \cite[Section~3.2]{duan2019cluster} in the following way.

In type $A_{2n-1}$, a shortened snake coincides exactly with an ordinary snake. Hence, in this type, we will use the term snake interchangeably.

In type $B_n$, given a shortened snake $(i_t, k_t)_{1 \leq t \leq T}$, we replace the terms $(i_t,k_t)$ in the sequence by two terms $(i_t,k_t-1)(i_t,k_t+1)$ whenever $i_t =n$, and retain all other terms unchanged. The resulting sequence is then a snake in the sense of \cite{mukhin2012path,duan2019cluster}. This gives an injective (but not surjective) map
\[\{\text{shortened snakes}\} \to \{\text{snakes}\}.\]

We use the notion of shortened snakes so that the weights of the snake module will be contained in the sublattice $P'$ of $P$, as we have seen in Remark~\ref{remark: P prime}.

\subsection{Non-overlapping paths and snake modules}
For two paths $p$ and $p'$ of the same type, $p$ is said being strictly above $p'$ (or, $p'$ strictly below $p$) if
\[
\forall (x,y) \in p, (x,z) \in p' \Rightarrow y < z.
\]
In this case, we denote by $p \succ p'$.

A $T$-tuple of paths $(p_1, \dots , p_T )$ is said to be \textit{non-overlapping} if $p_1 \succ p_2 \succ \cdots \succ p_T$.

For any shortened snake $\big((i_1,k_1)(i_2,k_2)\cdots(i_T,k_T)\big)$, $T\in\mathbb{N}^*$, define
\[
\overline{\mathscr{P}}_{(i_t,k_t)_{1\leq t \leq T}}:=\{(p_1,\dots,p_T)~|~p_t \in\mathscr{P}_{i_t,k_t},1\leq t \leq T,(p_1,\dots,p_T)~\text{is non-overlapping}\}.
\]
This is the set of $T$-tuples of \textit{non-overlapping paths} (NOP), where the $t$-th path is in the set $\mathscr{P}_{i_{t},k_{t}}$ determined by the snake.

The paths are combinatorial tools for describing the $q$-characters of a large class of representations of quantum affine algebras, called snake modules \cite{mukhin2012path}. 

\begin{definition}\label{def: snake modules}
    Let $\Uqghat$ be a quantum affine algebra of type $A_{2n-1}^{(1)}$ or $B_n^{(1)}$. Let $(i_t,k_t)_{1\leq t\leq T}$ be a shortened snake of the corresponding type. Let $L(m)$ be the irreducible $l$-highest weight representation of $\Uqghat$, with the highest $l$-weight
    \begin{equation}\label{eqn: m snake highest weight}
        \begin{split}
            & m = \prod_{1 \leq t \leq T} Y_{i_t,q^{k_t}},~\text{ in type $A_{2n-1}$}, \\
            & m = \prod_{t:i_t \neq n} Y_{i_t,q^{k_t}} \prod_{t: i_t = n} Y_{n,q^{k_t-1}}Y_{n,q^{k_t+1}},~\text{ in type $B_n$}. 
        \end{split}
    \end{equation}
    These are known as \textit{snake modules} of $\Uqghat$.
\end{definition}

The path description provides a method to compute both the usual characters and $q$-characters of snake modules of types $A_{2n-1}^{(1)}$ and $B_n^{(1)}$.

\begin{theorem}\cite[Theorem~6.1]{mukhin2012path}\label{path description}
Let $(i_t,k_t)_{1\leq t\leq T}$ be a shortened snake and $m$ as in \eqref{eqn: m snake highest weight}, then we have
\[
\chi_q\big( L(m) \big) = \mathop{\sum}\limits_{(p_1,\dots,p_T)\in\overline{\mathscr{P}}_{(i_t,k_t)_{1\leq t\leq T}}}\prod_{t=1}^{T}\mathfrak{m}(p_t),
\]
where 
\begin{equation}\label{eqn: monomial m of p}
\mathfrak{m}(p) = \prod_{(j,\ell) \in C_p^+} Y_{j,q^\ell} \prod_{(j,\ell) \in C_p^-} Y_{j,q^\ell}^{-1} \quad \text{for any path $p$.}
\end{equation}
\end{theorem}

Similarly, for a $T$-tuple of NOP $\overline{p} = (p_1,\dots,p_T)$, we denote by 
\begin{equation}\label{eqn: monomial m of non-overlapping paths}
    \mathfrak{m}(\overline{p}) = \prod_{t=1}^T \mathfrak{m}(p_t).
\end{equation}

By \cite[Theorem~3]{frenkel1999q}, the usual character $\chi$ can be obtained from the $q$-character directly.

\begin{corollary}\label{character-path}
 Let $(i_t,k_t)_{1\leq t\leq T}$ be a shortened snake and $m$ as in \eqref{eqn: m snake highest weight}, then the usual character
\[
\chi \big( L(m)\big) = \mathop{\sum}\limits_{(p_1,\dots,p_T)\in\overline{\mathscr{P}}_{(i_t,k_t)_{1\leq t\leq T}}}\bigg[\sum_{t=1}^{T}\mathfrak{m}'(p_t) \bigg],
\]
where 
\begin{equation}\label{eqn: m prime of p}
\mathfrak{m}'(p) = \sum_{(j,\ell) \in C_p^+} \omega_{j} - \sum_{(j,\ell) \in C_p^-} \omega_{j} \quad \text{for any path $p$.}
\end{equation}
\end{corollary}
We also write 
\[\mathfrak{m}'(\overline{p}) = \sum_{t=1}^T \mathfrak{m}'(p_t)\]
for a $T$-tuple of NOP $\overline{p} = (p_1,\dots,p_T)$.

\subsection{Path descriptions for twisted type \texorpdfstring{$A_{2n-1}^{(2)}$}{}}
This section is devoted to prove the following proposition.
\begin{proposition}\label{pro:snake modules conjecture}
    Let $L(m)$ be a snake module of the quantum affine algebra of type $A_{2n-1}^{(1)}$, let $\bar{\pi}$ be the folding map between the Grothendieck rings in \eqref{eqn:bar pi of Grothendieck rings} from type $A_{2n-1}^{(1)}$ to $A_{2n-1}^{(2)}$. Then we have 
    \[\bar{\pi} ([L(m)]) = [L(\pi(m))].\]
\end{proposition}

Recall that by the definition of $\bar{\pi}$, this is equivalent to say that $\pi(\chi_q(L(m))) = \chi_q^{\sigma}(L(\pi(m)))$. 

Since $\chi_q(L(m))$ has a path description 
\[\chi_q(L(m)) = \sum_{m' \in \mathcal{M}} m',\]
where $\mathcal{M}$ is the finite set of monomials $\{ \mathfrak{m}(\overline{p})~|~\overline{p} \in\overline{\mathscr{P}}_{(i_t,k_t)_{1\leq t\leq T}} \}$, Proposition~\ref{pro:snake modules conjecture} is equivalent to

\begin{proposition}
    Let $m$ and $\mathcal{M}$ be as above. Then
    \[\chi_q^{\sigma} (L(\pi(m))) = \sum_{\pi(m') \in \pi(\mathcal{M})} \pi(m'),\]
    where $\pi(\mathcal{M}) = \{ \pi \circ \mathfrak{m}(\overline{p})~|~\overline{p} \in\overline{\mathscr{P}}_{(i_t,k_t)_{1\leq t\leq T}} \}$ is a finite set of monomials in the ring $\mathcal{Z}$.
\end{proposition}
\begin{proof}
    The proof largely follows the arguments of \cite[Theorem~3.4]{mukhin2012path} and \cite[Theorem~6.1]{mukhin2012path}. We briefly recall the main steps, emphasizing the differences in the twisted case.

    \begin{itemize}
        \item The proof of \cite[Theorem~3.4]{mukhin2012path} used \cite[Proposition~3.3]{mukhin2012path}. The twisted analogue of this proposition is proved in \cite[Proposition 2.16]{Dahiya2025}.

        \item The condition (i) in \cite[Theorem~3.4]{mukhin2012path} is verified for $m$ and $\mathcal{M}$ by \cite[Lemma~5.14]{mukhin2012path}. Moreover, when restricting $\mathcal{M}$ to be a set of monomials in $\mathbb{Z}[Y_{i,q^n}]_{i \in I, n \in \ZZ}$, we have 
        \[\pi(m') \in \pi(\mathcal{M}) \text{ is  dominant} \iff m' \in \mathcal{M} \text{ is dominant}.\] Thus $\pi(m)$ and $\pi(\mathcal{M})$ also verify the condition (i).

        \item The condition (ii)  is verified for $\mathcal{M}$ by \cite[Lemma~5.12]{mukhin2012path}, which directly implies that it also holds for $\pi(\mathcal{M})$.

        \item The condition (iii) can be verified as follows:  when $\bar{i} = \bar{n}$, this is simply an $\mathfrak{sl}_2$-restriction and follows directly from \cite[Lemma~5.13]{mukhin2012path}. When $\bar{i} \neq \bar{n}$, it is known that for a $\mathcal{U}_q(\hat{\mathfrak{sl}}_2)$-module $L(m_1 m_2)$, where $m_1 \in \ZZ[Y_{q^n}]_{n \in \ZZ}$ and $m_2 \in \ZZ[Y_{-q^n}]_{n \in \ZZ}$, we have
        \[L(m_1m_2) \simeq L(m_1) \otimes L(m_2),\]
        and in particular
        \[\chi_q(L(m_1m_2)) = \chi_q(L(m_1))\chi_q(L(m_2)).\]
        For $m' = \prod_{i \in I, a \in \CC^*} Y_{i,a}^{u_{i,a}}$, we use the notation $-m' := \prod_{i \in I, a \in \CC^*} Y_{i,-a}^{u_{i,a}}$.
        Following the notation in \cite[Theorem~3.4]{mukhin2012path}, if $\pi(M) \in \pi(\mathcal{M})$ is $\bar{i}$-dominant, then $M$ is both $i$-dominant and $(2n-i)$-dominant, and
        \[\chi_q(L(\beta_{\bar{i}}(\pi(M)))) = \chi_q(L(\beta_i(M))) \chi_q(L(\beta_{2n-i}(-M))).\] 
        Condition (iii) is verified immediately from the computation 
        \[\begin{split}
            &\sum_{\pi(m') \in \pi(m)\ZZ[A^{\pm 1}_{\bar{i},\pm q^n}]_{n \in \ZZ} \cap \pi(\mathcal{M})} \beta_{\bar{i}}(\pi(m')) \\
            & = \sum_{\pi(m') \in \pi(m)\ZZ[A^{\pm 1}_{\bar{i},\pm q^n}]_{n \in \ZZ} \cap \pi(\mathcal{M})} \beta_{i}(m') \beta_{2n-i}(-m') \\
            & = \sum_{m' \in m\ZZ[A^{\pm 1}_{i,q^n}]_{n \in \ZZ} \cap \mathcal{M}} \beta_{i}(m') \times \sum_{m'' \in m\ZZ[A^{\pm 1}_{2n-i,q^n}]_{n \in \ZZ} \cap \mathcal{M}} \beta_{2n-i}(-m'').
        \end{split}
        \]
        The last equation follows from the fact that for a fixed $i < n$, if $m\prod_{a} A_{i,a}^{u_{a}} \in \mathcal{M}$ and $m\prod_{a} A_{2n-i,a}^{v_{a}} \in \mathcal{M}$, then $m\prod_{a} A_{i,a}^{u_{a}} A_{2n-i,a}^{v_{a}} \in \mathcal{M}$, where $u_a,v_a \in \ZZ$.
    \end{itemize} 
    In conclusion, the proof and results in \cite[Theorem~3.4]{mukhin2012path} and \cite[Theorem~6.1]{mukhin2012path} apply to $q$-characters of snake modules of twisted quantum affine algebras of type $A_{2n-1}^{(2)}$.
\end{proof}

For terminological convenience, we extend the definition of snake modules to twisted quantum affine algebra of type $A_{2n-1}^{(2)}$.
\begin{definition}
    A snake module of twisted quantum affine algebra $\Uqghatsigma$ of type $A_{2n-1}^{(2)}$ is an irreducible finite-dimensional $l$-highest weight representation $L(m)$, $m \in \mathcal{Z}$, such that 
    \[m = \prod_{1 \le t \le T} Z_{\bar{i_t},q^{k_t}},\]
    where $(i_t,k_t)_{1 \le t \le T}$ is a snake of type $A_{2n-1}$ such that $1 \le i_t \le n$ for all $t$.
\end{definition}

Recall that usual characters can be obtained directly from $q$-characters \cite[Theorem~3]{frenkel1999q}. We define the folding map on usual characters.

\begin{definition}
    Let $\varpi$ be the ring homomorphism \begin{equation}\label{eqn: folding character}
    \varpi : \ZZ[\bigoplus_{1 \le i \le 2n-1} \ZZ \omega_i] \to \ZZ[\bigoplus_{i \in \{\bar{1},\dots,\bar{n}\}} \ZZ \check{\omega}_i] =  \ZZ[\leftindex^L{P}],
    \end{equation}
which maps $[\sum_{i=1}^{2n-1} k_i\omega_{i}]$ to $[\sum_{i=1}^{n-1} (k_i+k_{2n-i})\check{\omega}_{\bar{i}} +k_n\check{\omega}_{\bar{n}}]$.
\end{definition}

As a consequence of Proposition~\ref{pro:snake modules conjecture}, we have
\begin{proposition}\label{prop: folding snake formula}
    Let $(i_t,k_t)_{1 \le t \le T}$ be a snake of type $A_{2n-1}$ such that $1 \le i_t \le n$ for all $t$. Let $L(\prod_{1 \le t \le T} Y_{i_t,q^{k_t}})$ be the corresponding snake module of type $A_{2n-1}^{(1)}$ and $L(\prod_{1 \le t \le T} Z_{\bar{i_t},q^{k_t}})$ be the corresponding snake module of type $A_{2n-1}^{(2)}$. Then
    \[\chi^{\sigma}(L(\prod_{1 \le t \le T} Z_{\bar{i_t},q^{k_t}})) = \varpi (\chi(L(\prod_{1 \le t \le T} Y_{i_t,q^{k_t}}))).\]
\end{proposition}

\section{Langlands dual from \texorpdfstring{$B_n^{(1)}$}{} to \texorpdfstring{$A_{2n-1}^{(2)}$}{}}\label{sec:Langlands B to A}
In this section, we work within type $A_{2n-1}$ and type $B_{n}$, $n\geq 2$. The goal of this section is to construct an injective map which associates a tuple of non-overlapping paths of type $A_{2n-1}$ with a tuple of non-overlapping paths of type $B_n$. As a consequence, we deduce that snake modules of type $B_n^{(1)}$ admit Langlands dual representations.

\subsection{The map between sets of points}
In this subsection, we begin with constructing the map on the level of points.

Define the set of points 
\[X^A = \{(j,\ell)\in \{0,\dots,2n\} \times \ZZ~|~\ell \equiv n + j + 1 \pmod 2 \}.\]

We construct a map $f$ from the set of points $X^A$ to the set of points $X^B$. Recall that 
\[X^B = \big( \{0,1,\dots,n-1\} \times 2\ZZ \big) \sqcup \big(\{n\} \times (2\ZZ+1)\big).\]

\begin{definition}\label{def: map f between sets of points}
Define an injective map
\begin{align*}
f: X^A &\rightarrow X^B\\
(j,\ell) &\mapsto \begin{cases}(j,2\ell), &\text{if}~j< n,\\
(n, 2\ell-1), &\text{if}~j= n,\\
(2n-j,2\ell-2), &\text{if}~j> n.
\end{cases}
\end{align*}
\end{definition}
This is injective because if $(j,\ell),(2n-j,\ell') \in X^A$, then $\ell \equiv \ell' \pmod 2$, thus $2\ell \not\equiv 2\ell'-2 \pmod 4$.

\subsection{The map between paths}
The goal of this subsection and the next subsection is to construct a map which associates a tuple of NOP of type $A_{2n-1}$ with a tuple of NOP of type $B_n$. We begin our construction for one single path.

\begin{definition}\label{def: map F between paths}
    Let $p \in \mathscr{P}^{A}_{i,k}$, $i \leq n$. We define a path $F(p)$ in $\mathscr{P}^{B}_{i,2k}$ in the following way.

    Write \[p = \big( (0,x_0),(1,x_1),\dots,(2n,x_{2n})  \big).\]

    Then we define 
    \begin{align*}
        F(p) = \big( (0,y_0), &(2,y_1), \dots, (2n-4, y_{n-2}), (2n-2,y_{n-1}), (2n-1,y_n), \\
         &(2n-1,z_n), (2n,z_{n-1}),\dots,(4n-4,z_1),(4n-2,z_0) \big),
    \end{align*}

    with coordinates 
    \[y_j = 
    \begin{cases}
        2x_j, &\text{if}~j < n,\\
        2x_n-1 + \epsilon, &\text{if}~j=n~\text{and}~x_n>x_{n-1}, \\
        2x_n+1-\epsilon, &\text{if}~j=n~\text{and}~x_n<x_{n-1},
    \end{cases}
    \]
    and
    \[z_j = 
    \begin{cases}
        2x_{2n-j}-2, &\text{if}~j < n,\\
        2x_n-3 + \epsilon, &\text{if}~j=n~\text{and}~x_n>x_{n+1}, \\
        2x_n-1-\epsilon, &\text{if}~j=n~\text{and}~x_n<x_{n+1}.
    \end{cases}
    \]
\end{definition}

    One can easily verify that $F(p)$ is indeed a path in $\mathscr{P}^{B}_{i,2k}$. An example of $p$ and $F(p)$ can be illustrated in Figure~\ref{illustrate of F}.

\begin{lemma}\label{lemma: F injective}
    The map $F: \mathscr{P}_{i,k}^A \to \mathscr{P}_{i,2k}^B$ is injective.
\end{lemma}
\begin{proof}
    This follows immediately from the definition of $F$.
\end{proof}

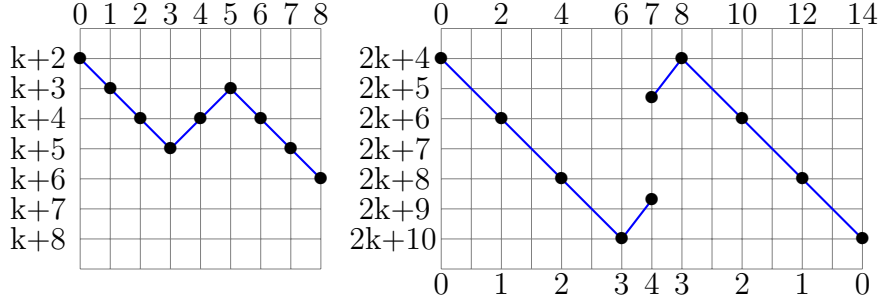
\begin{figure}[htp]
    \centering
\begin{tikzpicture}[scale=0.4]
\draw[step=1cm, gray, very thin] (0, 0) grid (8, 8);
\node at (0, 8.5) {0}; \node at (1, 8.5) {1};
\node at (2, 8.5) {2}; \node at (3, 8.5) {3};
\node at (4, 8.5) {4}; \node at (5, 8.5) {5};
\node at (6, 8.5) {6}; \node at (7, 8.5) {7};
\node at (8, 8.5) {8};

\node at (-1.4, 0) { };
\node at (-1.4, 1) {k+8};
\node at (-1.4, 2) {k+7};
\node at (-1.4, 3) {k+6};
\node at (-1.4, 4) {k+5}; \node at (-1.4, 5) {k+4};
\node at (-1.4, 6) {k+3}; \node at (-1.4, 7) {k+2};
\node at (-1.4, 8) { }; 

\draw[blue, thick] (0,7)-- (1,6)-- (2,5)--(3,4)--(4,5)--(5,6)--(6,5) --(7,4) -- (8,3);

\node at (0,7) {$\bullet$};
\node at (1,6){$\bullet$}; 
\node at (2,5){$\bullet$}; 
\node at (3,4) {$\bullet$}; 
\node at (4,5) {$\bullet$};
\node at (5,6) {$\bullet$}; 
\node at (6,5) {$\bullet$};
\node at (7,4) {$\bullet$};
\node at (8,3) {$\bullet$};

\begin{scope}[xshift=12cm]
\draw[step=1cm, gray, very thin] (0, 0) grid (14, 8);
\node at (0, 8.5) {0}; \node at (2, 8.5) {2};
\node at (4, 8.5) {4}; \node at (6, 8.5) {6};
\node at (7, 8.5) {7}; \node at (8, 8.5) {8};
\node at (10, 8.5) {10}; \node at (12, 8.5) {12};
\node at (14, 8.5) {14};

\node at (0, -0.5) {0}; \node at (2, -0.5) {1};
\node at (4, -0.5) {2}; \node at (6, -0.5) {3};
\node at (7, -0.5) {4}; \node at (8, -0.5) {3};
\node at (10, -0.5) {2}; \node at (12, -0.5) {1};
\node at (14, -0.5) {0};

\node at (-1.55, 0) { };
\node at (-1.6, 1) {2k+10};
\node at (-1.55, 2) {2k+9};
\node at (-1.55, 3) {2k+8};
\node at (-1.55, 4) {2k+7};
\node at (-1.55, 5) {2k+6};
\node at (-1.55, 6) {2k+5}; \node at (-1.55, 7) {2k+4};
\node at (-1.55, 8) { };  
\draw[blue, thick] (0,7)--(6,1)--(7,2.3);
\draw[blue, thick] (7,5.7)--(8,7)--(14,1);
\node at (0,7) {$\bullet$}; 
\node at (2,5) {$\bullet$};
\node at (4,3) {$\bullet$};
\node at (6,1) {$\bullet$}; 
\node at (7,2.3) {$\bullet$};
\node at (7,5.7) {$\bullet$};
\node at (8,7) {$\bullet$};
\node at (10,5) {$\bullet$};
\node at (12,3) {$\bullet$};
\node at (14,1) {$\bullet$}; 
\end{scope}

\end{tikzpicture}
\caption{A path $p \in \mathscr{P}^A_{2,k}$ of type $A_7$ (left), and its image $F(p) \in \mathscr{P}^B_{2,2k}$ of type $B_4$ (right). Horizontal coordinates relabeled by $\tau^{-1}$ are marked at the bottom.}
\label{illustrate of F} 
\end{figure}

\begin{proposition}\label{prop: path map}
Let $p$ be a path in $\mathscr{P}_{i,k}^A$, $i \leq n$. Let $f : X^A \to X^B$ be the map in Definition~\ref{def: map f between sets of points}, and $F(p)$ be the path in $\mathscr{P}_{i,2k}^B$ defined above. Then
\begin{align*}
C^{-}_{F(p)}=&\{f(j,\ell)~|~(j,\ell)\in C^{-}_{p}, j\neq n\}\sqcup \{(n,2\ell-1), (n,2\ell-3)~|~(n,\ell)\in C^{-}_{p}\}\sqcup\\
&\{(n,2\ell-3)~|~(n-1,\ell+1),(n,\ell),(n+1,\ell-1)\in p\},\\
C^{+}_{F(p)}=&\{f(j,\ell)~|~(j,\ell)\in C^{+}_{p}, j \neq n\}\sqcup \{(n,2\ell-1), (n,2\ell+1)~|~(n,\ell)\in C^{+}_{p}\}\sqcup \\
&\{(n,2\ell+1)~|~(n-1,\ell+1),(n,\ell),(n+1,\ell-1)\in p\}.
\end{align*}

\end{proposition}
\begin{proof}
    The part with first coordinate $j \neq n$ follows immediately from definition of the map $F$. We now calculate the upper and lower corners of $F(p)$ with first coordinate $n$.

    Let $(n-1,x_{n-1}), (n,\ell), (n+1,x_{n+1}) \in p$ be consecutive points of the path $p$.
    \begin{itemize}
        \item If $x_{n-1} = \ell +1$ and $x_{n+1} = \ell -1$, then we have four consecutive points in $F(p)$:
        \[(2n-2,2\ell+2), (2n-1,2\ell+1-\epsilon),(2n-1,2\ell-3+\epsilon),(2n,2\ell-4).\]
        Thus we are in the case $(n,2\ell +1) \in C_{F(p)}^+$ and $(n,2\ell -3) \in C_{F(p)}^-$.

        \item If $x_{n-1} = \ell -1$ and $x_{n+1} = \ell +1$, then we have four consecutive points in $F(p)$:
        \[(2n-2,2\ell-2), (2n-1,2\ell-1+\epsilon),(2n-1,2\ell-1-\epsilon),(2n,2\ell).\]
        Thus we are in the case that $C_{F(p)}^+$ and $C_{F(p)}^-$ have no point with first coordinate $n$.
        
        \item If $x_{n-1} = x_{n+1} = \ell +1$, then we have four consecutive points in $F(p)$:
        \[(2n-2,2\ell+2), (2n-1,2\ell+1-\epsilon),(2n-1,2\ell-1-\epsilon),(2n,2\ell).\]
        Thus we are in the case $(n,2\ell +1),(n,2\ell -1) \in C_{F(p)}^+$.
        
        \item If $x_{n-1} = x_{n+1} = \ell -1$, then we have four consecutive points in $F(p)$:
        \[(2n-2,2\ell-2), (2n-1,2\ell-1+\epsilon),(2n-1,2\ell-3+\epsilon),(2n,2\ell-4).\]
        Thus we are in the case $(n,2\ell -1),(n,2\ell -3) \in C_{F(p)}^-$.
    \end{itemize}
\end{proof}

Recall that $\varpi$ is the folding map in \eqref{eqn: folding character} and $\Pi$ is the Langlands dual map on the weight lattice in \eqref{eqn: Langlands character dual map}.

\begin{theorem}\label{thm: map between paths preserves the character}
    Let $\mathfrak{m}'$ be the associated weight of a path, as defined in \eqref{eqn: m prime of p}. Then for any path $p \in \mathscr{P}^A_{i,k}$, $i \le n$, we have  
    \[\varpi \circ \mathfrak{m}'(p)=\Pi \circ \mathfrak{m}'(F(p)).\]
\end{theorem}
\begin{proof}
    The associated weight of a path is given by the positions of upper corners and lower corners. Suppose that the path $p$ has an upper (resp. lower) corner at first coordinate $j$, $j \neq n$, which provides $\omega_j$ (resp. $-\omega_j$) in the weight. By Proposition~\ref{prop: path map}, $F(p)$ has a corresponding upper (resp. lower) corner at first coordinate $\min(j,2n-j)$.

    At the first coordinate $n$, if $(n,\ell)$ is an upper (resp. lower) corner of $p$, then by  Proposition~\ref{prop: path map}, $F(p)$ has double upper (resp. lower) corners at first coordinate $n$, which contribute $2\omega_n$ (resp. $-2\omega_n$) to the weight. If $(n,\ell)$ is neither an upper corner nor a lower corner of $p$, then $F(p)$ has exactly one upper corner and one lower corner at first coordinate $n$, which contribute $0$ to the weight.

    Note that $\varpi(\omega_j) = \varpi(\omega_{2n-j}) = \check{\omega}_{\bar{j}}$, and $\Pi(\omega_j) = \check{\omega}_{\bar{j}}$, $\forall j \neq n$, $\Pi(2\omega_n) = \check{\omega}_{\bar{n}}$, the identity follows immediately.
\end{proof}

\subsection{The map between non-overlapping paths}\label{subsec: map between non-overlapping paths}
Having constructed the map $F$ for a single path, now we construct the map for a tuple of NOP.

For each given shortened snake of type $A_{2n-1}$, we associate a shortened snake of type $B_n$ in the following way.

Let $\big( (i_1,k_1)\cdots(i_T,k_T) \big)$, $T \in \NN^*$, be a shortened snake of type $A_{2n-1}$ such that $1 \leq i_t \leq n$, $\forall 1\leq t\leq T$ (see Section~\ref{subsection:snakeposition}). By Definition~\ref{def: shortened snakes} 
\[\big( (i_1,2k_1)(i_2,2k_2)\cdots(i_T,2k_T) \big)\]
is a shortened snake of type $B_n$.

\begin{proposition}\label{prop: maps between non-overlapping paths}
Let $(p_1,p_2,\dots,p_T)\in\overline{\mathscr{P}}^A_{(i_t,k_t)_{1\leq t\leq T}}$ be a $T$-tuple of NOP. Then the $T$-tuple 
\[
\big(F(p_1),F(p_2),\dots,F(p_T) \big)
\]
is a $T$-tuple of NOP in $\overline{\mathscr{P}}^B_{(i_t,2k_t)_{1 \leq t \leq T}}$. Here $F(p_t) \in \mathscr{P}_{i_t,2k_t}^B$ is the image of $p_t$ under the map $F: \mathscr{P}_{i_t,k_t}^A \to \mathscr{P}_{i_t,2k_t}^B$ in Definition~\ref{def: map F between paths}.
\end{proposition}
\begin{proof}
    One only has to check that $F(p_{t-1}) \succ F(p_{t})$, $\forall 2 \leq t \leq T$. 

    Suppose that for $j < n$, $(2j,y_j^{(t-1)}) \in F(p_{t-1})$ and $(2j,y_j^{(t)}) \in F(p_t)$, then $(j,y_j^{(t-1)}/2) \in p_{t-1}$ and $(j,y_j^{(t)}/2) \in p_t$. Since $p_{t-1} \succ p_t$, we have $y_j^{(t-1)} < y_j^{(t)}$.

    Similarly, suppose that for $j < n$, $(4n-2-2j,z_j^{(t-1)}) \in F(p_{t-1})$ and $(4n-2-2j,z_j^{(t)}) \in F(p_t)$, then $(2n-j,z_j^{(t-1)}/2+1) \in p_{t-1}$ and $(j,z_j^{(t)}/2+1) \in p_t$. Since $p_{t-1} \succ p_t$, we have $z_j^{(t-1)} < z_j^{(t)}$.

    At first coordinate $2n-1$, let $(n,x_n^{(t-1)}) \in p_{t-1}$ and $(n,x_n^{(t)}) \in p_{t}$. Then $p_{t-1} \succ p_t$ implies that $x_n^{(t-1)} < x_n^{(t)}$. Thus $x_n^{(t-1)} \le x_n^{(t)}-2$ since $x_n^{(t-1)} \equiv x_n^{(t)} \pmod 2$. Therefore, suppose that $(2n-1,y_n^{(t-1)}),(2n-1,z_n^{(t-1)}) \in F(p_{t-1})$ and $(2n-1,y_n^{(t)}), (2n-1,z_n^{(t)}) \in F(p_t)$, then
    \[z_n^{(t-1)} < y_n^{(t-1)} \le 2x_n^{(t-1)} + 1 -\epsilon \le 2x_n^{(t)} - 4 + 1 -\epsilon < 2x_n^{(t)} - 3 + \epsilon \le z_n^{(t)} < y_n^{(t)}.\]

    Thus we have verified that $F(p_{t-1}) \succ F(p_{t})$, $\forall 2 \leq t \leq T$. 
\end{proof}

As a consequence of Theorem~\ref{thm: map between paths preserves the character} and Proposition~\ref{prop: maps between non-overlapping paths}, we have
\begin{corollary}\label{cor: F preserves characters}
    Let $\overline{p} = (p_1,\dots,p_T) \in \overline{\mathscr{P}}^A_{(i_t,k_t)_{1\leq t\leq T}}$ and $F(\overline{p}) = (F(p_1),\dots,F(p_T)) \in \overline{\mathscr{P}}^B_{(i_t,2k_t)_{1 \leq t \leq T}}$. Then
    \[\varpi \circ \mathfrak{m}'(\overline{p})=\Pi \circ \mathfrak{m}'(F(\overline{p})).\]
\end{corollary}

\subsection{Langlands dual representations}\label{subsec: Langlands dual from B to A}

Now we can prove that the conjecture of Frenkel and Hernandez \cite[Conjecture~2.2]{frenkel2011langlandsfinite} holds for snake modules from type $B_n^{(1)}$ to type $A_{2n-1}^{(2)}$. Let $\Uqghat$ be the quantum affine algebra of type $B_n^{(1)}$. Let $\Uqghatdual$ be the twisted quantum affine algebra of type $A_{2n-1}^{(2)}$.

\begin{theorem}\label{thm: langlands dual rep A to B}
     Let $V$ be any shortened snake module of $\Uqghat$, whose highest weight is $\lambda \in P$. There is a shortened snake module ${}^L{V}$ of $\Uqghatdual$ whose highest weight is $\Pi(\lambda) \in \leftindex^L{P}$, such that
    \[ \chi^{\sigma}({}^L{V}) \preceq \Pi(\chi(V)).\]
\end{theorem}

\begin{proof}
    Let $(i_t,k_t)_{1 \leq t \leq T}$ be a shortened snake of type $B_{n}$, and 
        \begin{equation}\label{eqn: snake weight type B}
            m = \prod_{t:i_t \neq n}Y_{i_t,q^{k_t}} \prod_{t: i_t=n}Y_{n,q^{k_t-1}}Y_{n,q^{k_t+1}}
        \end{equation}
    the corresponding dominant monomial.

    Then $(i_t,\frac{k_t}{2})_{1 \leq t \leq T}$ is a shortened snake of type $A_{2n-1}$ such that $1 \leq i_t \leq n$, $\forall 1\leq t\leq T$. Let 
        \begin{equation}\label{eqn: snake weight Lm}
            {}^L{m} = \prod_{1 \leq t \leq T} Z_{\bar{i_t},q^{k_t/2}}
        \end{equation}
    be the corresponding dominant monomial. 

    We prove that if $V = L(m)$, then we can choose ${}^L{V} = L({}^L{m})$.

    By Proposition~\ref{prop: folding snake formula} and Corollary~\ref{character-path},
    \[\chi^{\sigma}({}^L{V}) = \varpi(\chi(L(\prod_{1 \leq t \leq T} Y_{i_t,q^{k_t/2}}))) =  \mathop{\sum}\limits_{\overline{p} \in\overline{\mathscr{P}}^A_{(i_t,\frac{k_t}{2})_{1\leq t\leq T}}}  \bigg[ \varpi \circ \mathfrak{m}'(\overline{p}) \bigg],\]
    and by Corollary~\ref{cor: F preserves characters}, it equals to
    \[\mathop{\sum}\limits_{\overline{p} \in\overline{\mathscr{P}}^A_{(i_t,\frac{k_t}{2})_{1\leq t\leq T}}}  \bigg[ \Pi \circ \mathfrak{m}'(F(\overline{p})) \bigg].\]

    On the other hand, by Corollary~\ref{character-path},
    \[\Pi(\chi(V)) = \mathop{\sum}\limits_{\overline{p} \in\overline{\mathscr{P}}^B_{(i_t,k_t)_{1\leq t\leq T}}}\bigg[\Pi \circ \mathfrak{m}'(\overline{p}) \bigg].\]

    By Lemma~\ref{lemma: F injective},
    \[F: \overline{\mathscr{P}}^A_{(i_t,\frac{k_t}{2})_{1\leq t\leq T}}\to \overline{\mathscr{P}}^B_{(i_t,k_t)_{1\leq t\leq T}} \]
    is injective on every single path, thus injective on $T$-tuples of NOP. Therefore,
    \[ \chi^{\sigma}({}^L{V}) \preceq \Pi(\chi(V)).\]
\end{proof}

It is interesting to know the image of the map $F$. For this purpose, we introduce the following notion.

\begin{definition}\label{def: level}
    Recall that a path $p$ of type $B_n$ has two points $(2n-1,y_n)$, $(2n-1,z_n)$ with first coordinate $2n-1$, and $y_n > z_n$. Define the \textit{gap} of the path $p$ to be
    \[\mathrm{gap}(p) := \frac{1}{2} (\lfloor \frac{y_n+1}{2} \rfloor -  \lfloor \frac{z_n+3}{2} \rfloor).\]

    The gap of a $T$-tuple of NOP is the sum of the gap of each component:
    \[\mathrm{gap}\big((p_1,\dots,p_T)\big) = \sum_{t=1}^T \mathrm{gap}(p_t).\]
\end{definition}

\begin{figure}[htp]
    \centering
\begin{tikzpicture}[scale=0.4]
\draw[step=1cm, gray, very thin] (0, 0) grid (14, 11);
\node at (0, 11.5) {0}; \node at (2, 11.5) {2};
\node at (4, 11.5) {4}; \node at (6, 11.5) {6};
\node at (7, 11.5) {7}; \node at (8, 11.5) {8};
\node at (10, 11.5) {10}; \node at (12, 11.5) {12};
\node at (14, 11.5) {14};

\node at (0, -0.5) {0}; \node at (2, -0.5) {1};
\node at (4, -0.5) {2}; \node at (6, -0.5) {3};
\node at (7, -0.5) {4}; \node at (8, -0.5) {3};
\node at (10, -0.5) {2}; \node at (12, -0.5) {1};
\node at (14, -0.5) {0};

\node at (-1.5, 0) { }; \node at (-1.5, 1) {k+9};
\node at (-1.5, 2) {k+8}; \node at (-1.5, 3) {k+7};
\node at (-1.5, 5) {k+5};\node at (-1.5, 4) {k+6};
\node at (-1.5, 7) {k+3};\node at (-1.5, 6) {k+4};
\node at (-1.5, 9) {k+1};\node at (-1.5, 8) {k+2};
\node at (-1.5, 10) {k};\node at (-1.5, 11) { };

\draw[blue, thick] (0,5) -- (2,7)-- (4,9)--(6,7)--(7,8.3);
\draw[blue, thick] (7,10.3)--(8,9) --(10,7)--(12,9) -- (14,7);

\draw[blue, thick] (0,1) --(2,3)-- (4,5)--(6,3)--(7,1.7);
\draw[blue, thick] (7,6.3)--(8,5) --(10,3)--(12,5) -- (14,3);

\node at (0,5) {$\bullet$}; \node at (2,7){$\bullet$}; 
\node at (4,9){$\bullet$}; \node at (6,7) {$\bullet$}; \node at (7,8.3) {$\bullet$};\node at (7,10.3) {$\bullet$}; \node at (8,9) {$\bullet$}; \node at (10,7) {$\bullet$};
\node at (12,9) {$\bullet$}; \node at (14,7) {$\bullet$};

\node at (0,1) {$\bullet$}; \node at (2,3){$\bullet$}; 
\node at (4,5){$\bullet$}; \node at (6,3) {$\bullet$}; 
\node at (7,1.7) {$\bullet$};\node at (7,6.3){$\bullet$}; 
\node at (8,5) {$\bullet$}; \node at (10,3){$\bullet$};
\node at (12,5) {$\bullet$}; \node at (14,3) {$\bullet$};

\begin{scope}[xshift=19cm]
\draw[step=1cm, gray, very thin] (0, 0) grid (14, 11);
\node at (0, 11.5) {0}; \node at (2, 11.5) {2};
\node at (4, 11.5) {4}; \node at (6, 11.5) {6};
\node at (7, 11.5) {7}; \node at (8, 11.5) {8};
\node at (10, 11.5) {10}; \node at (12, 11.5) {12};
\node at (14, 11.5) {14};

\node at (0, -0.5) {0}; \node at (2, -0.5) {1};
\node at (4, -0.5) {2}; \node at (6, -0.5) {3};
\node at (7, -0.5) {4}; \node at (8, -0.5) {3};
\node at (10, -0.5) {2}; \node at (12, -0.5) {1};
\node at (14, -0.5) {0};

\node at (-1.5, 0) { }; \node at (-1.5, 1) {k+9};
\node at (-1.5, 2) {k+8}; \node at (-1.5, 3) {k+7};
\node at (-1.5, 5) {k+5};\node at (-1.5, 4) {k+6};
\node at (-1.5, 7) {k+3};\node at (-1.5, 6) {k+4};
\node at (-1.5, 9) {k+1};\node at (-1.5, 8) {k+2};
\node at (-1.5, 10) {k};\node at (-1.5, 11) { };

\draw[blue, thick] (0,5) -- (2,7)-- (4,9)--(6,7)--(7,5.7);
\draw[blue, thick] (7,10.3)--(8,9) --(10,7)--(12,9) -- (14,7);

\draw[blue, thick] (0,1) --(2,3)-- (4,5)--(6,3)--(7,1.7);
\draw[blue, thick] (7,2.3)--(8,1) --(10,3)--(12,5) -- (14,3);

\node at (0,5) {$\bullet$}; \node at (2,7){$\bullet$}; 
\node at (4,9){$\bullet$}; \node at (6,7) {$\bullet$}; \node at (7,5.7) {$\bullet$};\node at (7,10.3) {$\bullet$}; \node at (8,9) {$\bullet$}; \node at (10,7) {$\bullet$};
\node at (12,9) {$\bullet$}; \node at (14,7) {$\bullet$};

\node at (0,1) {$\bullet$}; \node at (2,3){$\bullet$}; 
\node at (4,5){$\bullet$}; \node at (6,3) {$\bullet$}; 
\node at (7,1.7) {$\bullet$}; \node at (7,2.3) {$\bullet$}; 
\node at (8,1) {$\bullet$}; \node at (10,3){$\bullet$};
\node at (12,5) {$\bullet$}; \node at (14,3) {$\bullet$};

\end{scope}

\end{tikzpicture}
\caption{Two examples of $2$-tuple of NOP of gap $1$.}
\label{illustrate of gap} 
\end{figure}
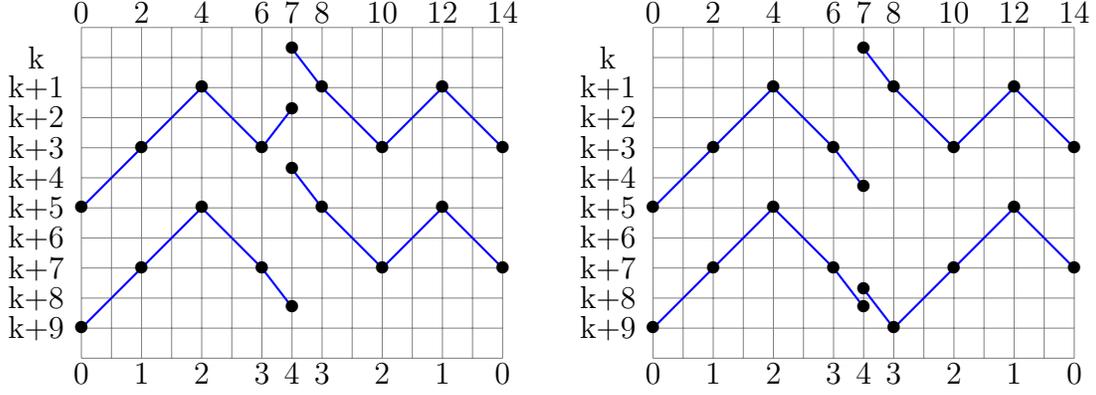

\begin{remark}
    The gap is a non-negative integer. This is because one can prove that 
    $$\frac{1}{2} (\lfloor \frac{y_n+1}{2} \rfloor -  \lfloor \frac{z_n+3}{2} \rfloor) = \lfloor \frac{y_n-z_n}{4} \rfloor,$$
    using the properties $y_n > z_n$ and
    \[y_n \in (4\ZZ +3-\epsilon) \sqcup (4\ZZ+1+\epsilon),\]
    \[z_n \in (4\ZZ +3+\epsilon) \sqcup (4\ZZ+1-\epsilon).\]
    We omit the calculation.
\end{remark}

\begin{lemma}
    A $T$-tuple of NOP in the image of $F$ has gap $0$.
\end{lemma}
\begin{proof}
    This follows from Definition~\ref{def: map F between paths}. Note that for a type $A_{2n-1}$ path $p$, if $(n,x_n) \in p$ and $(2n-1,y_n),(2n-1,z_n) \in F(p)$, then
    \[0 \le \frac{1}{2} (\lfloor \frac{y_n+1}{2} \rfloor -  \lfloor \frac{z_n+3}{2} \rfloor) \le \frac{1}{2}(\lfloor \frac{2x_n+1-\epsilon+1}{2} \rfloor -  \lfloor \frac{2x_n-3+\epsilon+3}{2} \rfloor) = 0.\]
    
\end{proof}

\begin{proposition}\label{prop: level 0 bijection}
    Let $(p_1,\dots,p_T) \in \overline{\mathscr{P}}^B_{(i_t,k_t)_{1\leq t \leq T}}$ be a $T$-tuple of NOP. Then $(p_1,\dots,p_T)$ is in the image $F(\overline{\mathscr{P}}^A_{(i_t,\frac{k_t}{2})_{1\leq t \leq T}})$ if and only if it has gap $0$. 
\end{proposition}
\begin{proof}
    It remains to prove that for a single path $p \in \mathscr{P}^B_{i,k}$, if $p$ has gap $0$, then $p$ is in the image $F(\mathscr{P}^A_{i,\frac{k}{2}})$.

    In fact, we are able to write down its preimage explicitly. Let
        \begin{align*}
           && p = \big( (0,y_0), &(2,y_1), \dots, (2n-4, y_{n-2}), (2n-2,y_{n-1}), (2n-1,y_n), \\
           && &(2n-1,z_n), (2n,z_{n-1}),\dots,(4n-4,z_1),(4n-2,z_0) \big).
        \end{align*}
    Then its preimage is the path 
        \begin{align*}
           && F^{-1}(p) = \big( (0,\frac{y_0}{2}), &(1,\frac{y_1}{2}), \dots, (n-1, \frac{y_{n-1}}{2}), (n,\lfloor \frac{y_{n}+1}{2} \rfloor ), \\
           && &(n+1,\frac{z_{n-1}+2}{2}),\dots,(2n-1,\frac{z_1+2}{2}),(2n,\frac{z_0+2}{2}) \big).
        \end{align*}
    It is straightforward to check that $F^{-1}(p)$ is a path in $\mathscr{P}^A_{i,\frac{k}{2}}$. One only has to use that $y_i,z_i$ are even integers when $i \neq n$, and $y_n \in 2\ZZ +1 \pm \epsilon$.

    Moreover, for a $T$-tuple of NOP $(p_1,\dots,p_T) \in \overline{\mathscr{P}}^B_{(i_t,k_t)_{1\leq t \leq T}}$, the corresponding paths $F^{-1}(p_1),\dots,F^{-1}(p_T)$ are also non-overlapping. In fact, for points whose first coordinate is not $2n-1$, the inequality is evident. For points with first coordinate $2n-1$, if $p_{t-1}$ is strictly above $p_{t}$ and $(2n-1,y_n^{(t-1)}),(2n-1,z_n^{(t-1)}) \in p_{t-1}$ (resp. $(2n-1,y_n^{(t)}),(2n-1,z_n^{(t)}) \in p_{t}$) are points of $p_{t-1}$ (resp. $p_{t}$), then $z_n^{(t-1)} < y_n^{(t-1)} < z_n^{(t)} < y_n^{(t)}$. Since $y_n^{(t-1)},z_n^{(t)},y_n^{(t)} \in 2\ZZ + 1 \pm \epsilon$, we have $y_n^{(t)} - y_n^{(t-1)} \geq 2$, thus $\lfloor \frac{y_{n}^{(t-1)}+1}{2} \rfloor < \lfloor \frac{y_{n}^{(t)}+1}{2} \rfloor$.
\end{proof}

Since the path description gives not only the usual character but also the $q$-character, we have constructed the embedding at the level of $q$-characters. In \cite{frenkel2011langlandsfinite}, Frenkel and Hernandez provided an algorithm to compute which term in $\chi_q(V)$ corresponds to a term in $\chi_q^{\sigma}({}^L{V})$ when $V$ is a Kirillov-Reshetikhin module. Their method is known as the interpolating $(q,t)$-characters. In this paper, we propose an alternative method for this question, and more generally for snake modules.

The gap of a monomial in $\chi_q(V)$ is defined as the gap of its corresponding path. Then we have the following method to write down the $q$-character of $L({}^L{m})$.

\begin{corollary}
    Let $L(m)$ be a shortened snake module of $\Uqghat$ with $m$ in \eqref{eqn: snake weight type B}. Let $C_0$ be the set of $T$-tuples of NOP in $\overline{\mathscr{P}}^B_{(i_t,k_t)_{1\leq t \leq T}}$ having gap $0$. Then
    \[\chi_q^{\sigma}(L({}^L{m})) = \sum_{\overline{p} \in C_0} \pi \circ \mathfrak{m}(F^{-1}(\overline{p})),\]
    where $\mathfrak{m}$ is defined in \eqref{eqn: monomial m of non-overlapping paths}, and $\pi$ is the folding map on $q$-characters defined in \eqref{eqn: folding pi}.
\end{corollary}

\section{An identity of characters of type \texorpdfstring{$A_{n-1}^{(1)}$}{} snake modules}\label{sec:equality}
We prove an identity between characters of snake modules of type $A_{n-1}^{(1)}$. This identity will play a key role in the next section, where we establish the Langlands branching rule. In Section~\ref{subsec: path interpretation of identity}, we provide an interpretation of this identity in the language of paths.

\subsection{Main theorem}\label{subsec: main theorem}
In this section, we present a snake module in terms of multi-segments for simplicity.

A multi-segment is a finite family of segments $[l_j,r_j]$, such that $l_j, r_j \in \ZZ$, $l_j \le r_j$, for all $j$ in a finite index set.
    
A \textit{type $A_{n-1}$ snake} is a multi-segment such that the segments are ordered as $[l_t,r_t]$, $1 \leq t \leq T$, $T \in \NN^*$, which satisfies
\[l_1 < l_2 < \cdots < l_T, r_1 < r_2< \cdots < r_T,~\text{and}~0 \le r_t - l_t \le n, \forall t.\]
The number $T$ is called the length of the snake.

Recall that in type $A_{n-1}$, there is a correspondence between multi-segments and dominant monomials in $\ZZ[Y^{\pm1}_{i,q^{i+2m}}]_{i \in I, m \in \ZZ}$ \cite{chari1996quantum}, which associates an individual segment $[l,r]$ with $Y_{r-l,q^{r+l}}$, and associates the multi-segment with the product $\prod_{t=1}^T Y_{r_t-l_t,q^{r_t+l_t}}$. Here the variable $Y_{r-l,q^{r+l}}$ should be understood as $1$ if $r-l = 0$ or $n$.

Under this correspondence, it is easy to see that $[\mathbf{l},\mathbf{r}] = ([l_t,r_t])_{1 \le t \le T}$ is a type $A_{n-1}$ snake if and only if the sequence of points $(r_t-l_t,r_t+l_t) \in \{0,\dots,n\} \times \ZZ$, $1 \leq t \leq T$, forms a shortened snake in the sense of Definition~\ref{def: definition of snakes}.

The corresponding snake module is denoted by $L([\mathbf{l},\mathbf{r}])$.

\begin{definition}
    Given two sequences $\mathbf{l} = (l_1,\dots,l_T)$ and $\mathbf{l}' = (l'_1,\dots,l'_T)$, 
    \begin{itemize}
        \item 
        we say that 
        $$\mathbf{l} \leq \mathbf{l}', \text{ if } l_t \leq l_t', \forall t.$$

        \item In case that $\mathbf{l} \le \mathbf{l}'$, we denote by
        \[|\mathbf{l}'-\mathbf{l}| = \sum_{t=1}^T (l_t' - l_t) \in \NN.\]
    
        \item We say that 
        $$\mathbf{l}[-1] < \mathbf{l}', \text{ if } l_{t-1} < l'_t, \forall t,$$
        where $l_0 = -\infty$.

        Similarly, we say that 
        $$\mathbf{l}' < \mathbf{l}[1], \text{ if } l'_{t} < l_{t+1}, \forall t,$$
        where $l_{T+1} = \infty$.
    \end{itemize}
\end{definition}

This section is devoted to proving the following theorem.

\begin{theorem}\label{thm: type A equality}
Let $[\mathbf{l},\mathbf{r}]$ be a type $A_{n-1}$ snake of length $T$. For each fixed integer $M \in \NN$, we have
    \begin{equation}\label{eqn: type A equation}
    \sum_{\substack{\mathbf{l} \le \mathbf{l}' <  \mathbf{l}[1]\\ |\mathbf{l}'-\mathbf{l}| = M\\ [\mathbf{l}',\mathbf{r}]~\text{is a snake}}} \chi(L(\mathbf{l}',\mathbf{r})) = \sum_{\substack{\mathbf{r}[-1] < \mathbf{r}' \le \mathbf{r} \\|\mathbf{r}-\mathbf{r}'| = M\\ [\mathbf{l},\mathbf{r}']~\text{is a snake}}} \chi(L(\mathbf{l},\mathbf{r}')),
    \end{equation}
    where the sum on the left-hand side is taken over $\mathbf{l}'$ such that $[\mathbf{l'},\mathbf{r}]$ is a type $A_{n-1}$ snake and $\mathbf{l} \le \mathbf{l}' <  \mathbf{l}[1]$, $|\mathbf{l}'-\mathbf{l}| = M$. Similarly, the sum on the right-hand side is taken over $\mathbf{r}'$ such that $[\mathbf{l},\mathbf{r}']$ is a type $A_{n-1}$ snake and $\mathbf{r}[-1] < \mathbf{r}' \le  \mathbf{r}$, $|\mathbf{r} - \mathbf{r}'| =M$. 
\end{theorem}

\subsection{Determinant formula}
In this subsection, we recall the determinant formula for snake modules, which allows us to express the character of a snake module as an algebraic combination of characters of fundamental representations. We refer the reader to \cite{brito2024alternating} for the notation and results needed here. This formula can also be derived via Schur--Weyl duality from the corresponding result for algebraic groups \cite{tadic1995on, chenevier2008characters, lapid2014on, lapid2018geometric, bittmann2023on}.

We use the notation $W(l,r)$ to denote the character of $W(\bm{\omega}_{l,r})$ in \cite{brito2024alternating}. That is,
\begin{definition}
     For $l,r \in \ZZ$,
\begin{align*}
    W(l,r) = 
    \begin{cases}
    \chi(L([l,r])),&~\text{if}~0 \le r-l \le n, \\
    0,&~\text{otherwise}.
    \end{cases}
\end{align*}
\end{definition}

Notice that $[l,r]$ is a type $A_{n-1}$ snake if and only if $[l+m,r+m]$ is as well, and the usual $\g$-character of a fundamental is independent of the spectral parameter. Therefore, we have

\begin{lemma}\label{lemma: shift invariant}
    $W(l,r) = W(l+m,r+m)$, $\forall l,r,m \in \ZZ$.
\end{lemma}

The following lemma is a direct corollary of the determinant formula \cite[Theorem~3]{brito2024alternating}, where the formula there is stated as an equation in the Grothendieck ring of certain category of finite-dimensional representations of $\Uqghat$. Here we only need the resulting formula for their usual $\g$-characters.

\begin{lemma}[{\cite[Theorem~3]{brito2024alternating}}]\label{lemma: determinant formula}

Let $[\mathbf{l},\mathbf{r}]$ be a type $A_{n-1}$ snake of length $T$. Let $A$ be the $T \times T$ matrix whose $(i,j)$-entry is $W(l_i,r_j)$, then
    \begin{equation}
\chi(L([\mathbf{l},\mathbf{r}])) = \det (A) = \sum_{\sigma \in \Sigma_T} (-1)^{\sgn (\sigma)} \prod_{t=1}^T W(l_t,r_{\sigma(t)}),
    \end{equation}
    where $\Sigma_T$ is the permutation group of $T$ elements.
\end{lemma}

\begin{remark}
    The matrix $A$ is slightly different from the matrix $A(\mathbf{s})$ in \cite{brito2024alternating}, but it is easy to see that their determinants are equal.
\end{remark}

\begin{definition}
    For integers $l_1,\dots,l_T,r_1,\dots,r_T \in \ZZ$, we use the notation 
    \begin{equation}\label{eqn: definition of W}
    W(\mathbf{l},\mathbf{r}) = \sum_{\sigma \in \Sigma_T} (-1)^{\sgn (\sigma)} \prod_{t=1}^T W(l_t,r_{\sigma(t)}).\end{equation}
\end{definition}

Lemma~\ref{lemma: determinant formula} states that if $[\mathbf{l},\mathbf{r}]$ is a type $A_{n-1}$ snake, then 
\begin{equation}\label{eqn: chi = W}
    \chi(L([\mathbf{l},\mathbf{r}])) = W(\mathbf{l},\mathbf{r}).
\end{equation}

\begin{corollary}\label{cor: W = 0 when two are equal}
    If $l_i = l_j$ for some $1 \leq i \neq j \leq T$, then $W(\mathbf{l},\mathbf{r}) = 0$. Similarly, if $r_i = r_j$ for some $1 \leq i \neq j \leq T$, then $W(\mathbf{l},\mathbf{r}) = 0$. 
\end{corollary}
\begin{proof}
    The determinant is $0$ when two rows or two columns are identical.
\end{proof}

\subsection{Proof of Theorem~\ref{thm: type A equality}}\label{sebsec: proof of identity}
\subsubsection*{Step 1}\label{step 1}

By \eqref{eqn: chi = W}, the equality~\eqref{eqn: type A equation} is equivalent to  
\begin{equation}\label{eqn: W equation}
    \sum_{\substack{\mathbf{l} \le \mathbf{l}' <  \mathbf{l}[1]\\ |\mathbf{l}'-\mathbf{l}| = M\\ [\mathbf{l}',\mathbf{r}]~\text{is a snake}}} W(\mathbf{l}',\mathbf{r}) = \sum_{\substack{\mathbf{r}[-1] < \mathbf{r}' \le \mathbf{r} \\|\mathbf{r}-\mathbf{r}'| = M\\ [\mathbf{l},\mathbf{r}']~\text{is a snake}}} W(\mathbf{l},\mathbf{r}'),
\end{equation}

Notice that if $[\mathbf{l},\mathbf{r}]$ is a type $A_{n-1}$ snake, then a sequence $\mathbf{l}' = (l_1',\dots,l_T')$ which verifies $\mathbf{l} \le \mathbf{l}' < \mathbf{l}[1]$ is strictly increasing, because $l_t' < l_{t+1} \le l_{t+1}'$, $\forall 1 \le t \le T-1$. Moreover, $r_t - l_t' \le r_t - l_t \le n$, $\forall 1 \le t \le T$.

Therefore, if a multi-segment $[\mathbf{l}',\mathbf{r}]$ is not a type $A_{n-1}$ snake, then there exists a $t$, $1 \le t \le T$, such that $l_t' > r_t$. We claim that $W(\mathbf{l}',\mathbf{r}) = 0$ in this case.

In fact, the entries of the matrix $A$ verify that $A_{ij} = W(l_i',r_j) = 0$ for all $i \ge t$ and $j \le t$, because $l_i' \ge l_t' > r_t \ge r_j$. Any matrix satisfying this condition has determinant $0$.

Similarly, if a multi-segment $[\mathbf{l},\mathbf{r}']$ is not a type $A_{n-1}$ snake, then $W(\mathbf{l},\mathbf{r}') = 0$.

In conclusion, we can drop the condition on snakes in the summation in \eqref{eqn: W equation}, and the equation becomes equivalent to 
\[\sum_{\substack{\mathbf{l} \le \mathbf{l}' <  \mathbf{l}[1]\\ |\mathbf{l}'-\mathbf{l}| = M}} W(\mathbf{l}',\mathbf{r}) = \sum_{\substack{\mathbf{r}[-1] < \mathbf{r}' \le \mathbf{r}\\ |\mathbf{r} - \mathbf{r}'|= M}} W(\mathbf{l},\mathbf{r}'),\]
where $[\mathbf{l},\mathbf{r}]$ is a fixed type $A_{n-1}$ snake, and the sum is taken for sequences $\mathbf{l}'$ (resp. $\mathbf{r'}$).

\subsubsection*{Step 2}\label{step 2}
We prove that for any fixed sequences $\mathbf{l}$ and $\mathbf{r}$, which are not necessarily increasing, we always have 
\[\sum_{\substack{\mathbf{l} \leq  \mathbf{l}'\\ |\mathbf{l}'-\mathbf{l}| = M}} W(\mathbf{l}',\mathbf{r}) = \sum_{\substack{\mathbf{r}' \leq \mathbf{r}\\ |\mathbf{r} - \mathbf{r}'| = M}} W(\mathbf{l},\mathbf{r}'),\]
where $M \in \NN$, and the sums are taken over $\mathbf{l}'$ and $\mathbf{r}'$, respectively.

We calculate that
\[
\begin{split}
    \sum_{\substack{\mathbf{l} \leq  \mathbf{l}'\\ |\mathbf{l}' - \mathbf{l}| = M}} W(\mathbf{l}',\mathbf{r}) &= \sum_{\substack{\mathbf{l} \leq  \mathbf{l}'\\ |\mathbf{l}' - \mathbf{l}| = M}} \sum_{\sigma \in \Sigma_T} (-1)^{\sgn(\sigma)} \prod_{i=1}^T W(l_i',r_{\sigma(i)}) \\
    &= \sum_{\substack{a_1,\dots,a_T \in \NN\\ a_1 + \cdots + a_T = M}} \sum_{\sigma \in \Sigma_T} (-1)^{\sgn(\sigma)} \prod_{i=1}^T W(l_i + a_i,r_{\sigma(i)}).
    \end{split}\]
By Lemma~\ref{lemma: shift invariant}, this is equal to
    \[\begin{split}
    &\sum_{\substack{a_1,\dots,a_T \in \NN\\ a_1 + \cdots + a_T = M}} \sum_{\sigma \in \Sigma_T} (-1)^{\sgn(\sigma)} \prod_{i=1}^T W(l_i,r_{\sigma(i)}-a_i)\\
     &= \sum_{\sigma \in \Sigma_T} (-1)^{\sgn(\sigma)} \sum_{\substack{a_1,\dots,a_T \in \NN\\ a_1 + \cdots + a_T = M}} \prod_{i=1}^T W(l_i,r_{\sigma(i)}-a_i)\\
     &= \sum_{\sigma \in \Sigma_T} (-1)^{\sgn(\sigma)} \sum_{\substack{a_{\sigma(1)},\dots,a_{\sigma(T)} \in \NN\\ a_{\sigma(1)} + \cdots a_{\sigma(T)} = M}} \prod_{i=1}^T W(l_i,r_{\sigma(i)}-a_{\sigma(i)})\\
     &= \sum_{\sigma \in \Sigma_T} (-1)^{\sgn(\sigma)} \sum_{\substack{a_1,\dots,a_T \in \NN\\ a_1 + \cdots + a_T = M}} \prod_{i=1}^T W(l_i,r_{\sigma(i)}-a_{\sigma(i)})\\
    &=\sum_{\substack{a_1,\dots,a_T \in \NN\\ a_1 + \cdots + a_T = M}} \sum_{\sigma \in \Sigma_T} (-1)^{\sgn(\sigma)} \prod_{i=1}^T W(l_i,r_{\sigma(i)}-a_{\sigma(i)})\\   
      &= \sum_{\substack{\mathbf{r}'\leq \mathbf{r}\\ |\mathbf{r} - \mathbf{r}'| = M}} \sum_{\sigma \in \Sigma_T} (-1)^{\sgn(\sigma)} \prod_{i=1}^T W(l_i,r_{\sigma(i)}') =\sum_{\substack{\mathbf{r}'\leq \mathbf{r}\\ |\mathbf{r} - \mathbf{r}'| = M}} W(\mathbf{l},\mathbf{r}').
\end{split}
\]
This completes \nameref{step 2}.

\subsubsection*{Step 3}\label{step 3}
We show that when $[\mathbf{l},\mathbf{r}]$ is a type $A_{n-1}$ snake, then
\[\sum_{\substack{\mathbf{l} \le \mathbf{l}' <  \mathbf{l}[1]\\ |\mathbf{l}' - \mathbf{l}| = M}} W(\mathbf{l}',\mathbf{r}) = \sum_{\substack{\mathbf{l} \le  \mathbf{l}'\\ |\mathbf{l}' - \mathbf{l}| = M}} W(\mathbf{l}',\mathbf{r}),\]
and
\[\sum_{\substack{\mathbf{r}[-1] <  \mathbf{r}' \le \mathbf{r}\\ |\mathbf{r} - \mathbf{r}'| = M}} W(\mathbf{l},\mathbf{r}') = \sum_{\substack{\mathbf{r}' \leq  \mathbf{r}\\ |\mathbf{r} - \mathbf{r}'| = M}} W(\mathbf{l},\mathbf{r}').\]

We prove the first equality, as the second is similar. We need the following lemma.
\begin{lemma}\label{lemma: sum of W = 0 when two are equal}
    Let $\mathbf{p}$, $\mathbf{q}$ be two sequences of integers of length $T$. If $p_i=p_j$ for some $1 \le i \neq j \le T$, then
    \[\sum_{\substack{\mathbf{p} \le  \mathbf{p}'\\|\mathbf{p}'-\mathbf{p}|=M}} W(\mathbf{p}',\mathbf{q}) = 0.\]
\end{lemma}
\begin{proof}
    By the equality proved in \nameref{step 2},
    \[\sum_{\substack{\mathbf{p} \le  \mathbf{p}'\\|\mathbf{p}'-\mathbf{p}|=M}} W(\mathbf{p}',\mathbf{q}) = \sum_{\substack{\mathbf{q}' \le  \mathbf{q}\\ |\mathbf{q}-\mathbf{q}'|=M}} W(\mathbf{p},\mathbf{q}') .\]
    By Corollary~\ref{cor: W = 0 when two are equal}, each term on the right-hand side is $0$.
\end{proof}

Now we are ready to complete \nameref{step 3}. We use the following notation. Let $J = (j_1,\dots,j_T)$ be a sequence such that $j_t \in \{t,t+1\} \cap \{1,2,\dots,T\}$. Denote by $\mathrm{h}(J) = \sum_{t=1}^T (j_t -t)$.

Recall that now we have $l_1 < l_2 < \cdots < l_T$. Basic set theory tells us that for any function $f(\mathbf{l}')$ depending on the sequence $\mathbf{l}'$, we have
\[\sum_{\substack{\mathbf{l} \le \mathbf{l}' < \mathbf{l}[1]\\|\mathbf{l}'-\mathbf{l}|=M}} f(\mathbf{l}') = \sum_J (-1)^{\mathrm{h}(J)} \sum_{\substack{(l_{j_1},l_{j_2},\dots,l_{j_T}) \le \mathbf{l}'\\ |\mathbf{l}'-\mathbf{l}|=M}} f(\mathbf{l}').\]

When $f(\mathbf{l}') = W(\mathbf{l}',\mathbf{r})$, the term for $J = (1,2,\dots,T)$ on the right-hand side equals to $$\sum_{\substack{\mathbf{l} \le \mathbf{l}'\\|\mathbf{l}'-\mathbf{l}|=M}} W(\mathbf{l}',\mathbf{r}),$$ and all other terms corresponding to $J \neq (1,2,\dots,T)$ are zero by Lemma~\ref{lemma: sum of W = 0 when two are equal}, because at least two indices $j_s,j_t$ are equal, $1 \le s \neq t \le T$.

Combining \nameref{step 2} and \nameref{step 3}, we have
\[\sum_{\substack{\mathbf{l} \le \mathbf{l}' <  \mathbf{l}[1]\\ |\mathbf{l}' - \mathbf{l}| = M}}  W(\mathbf{l}',\mathbf{r}) = \sum_{\substack{\mathbf{l} \le  \mathbf{l}'\\ |\mathbf{l}' - \mathbf{l}| = M}}  W(\mathbf{l}',\mathbf{r}) = \sum_{\substack{\mathbf{r}' \le  \mathbf{r}\\ |\mathbf{r} - \mathbf{r}'| = M}} W(\mathbf{l},\mathbf{r}') = \sum_{\substack{\mathbf{r}[-1] <  \mathbf{r}' \le \mathbf{r}\\ |\mathbf{r} - \mathbf{r}'| = M}}W(\mathbf{l},\mathbf{r}').\]

By the conclusion in \nameref{step 1}, we have completed the proof of Theorem~\ref{thm: type A equality}.

\subsection{Paths interpretation}\label{subsec: path interpretation of identity}
In this paper, we will need Theorem~\ref{thm: type A equality} in the language of paths. In this subsection, we derive a corollary of Theorem~\ref{thm: type A equality} formulated in terms of path descriptions.

Type $A_{2n-1}$ paths were defined in Definition~\ref{def: paths}. In this subsection and the following section, we will need type $A_{n-1}$ paths for all $n \ge 2$. While the definition is analogous to Definition~\ref{def: paths}, we present it here for completeness with a minor extension that permits $i = 0$ or $i = n$.

\begin{definition}
    For $0 \le i \le n$ and $k \in \ZZ$, a path of type $A_{n-1}$ is a finite sequence $\big( (0,y_0),\dots,(n,y_n) \big)$ in the set 
        \begin{align*}
           \mathscr{P}_{i,k}^{A_{n-1}}:=\{\big( (0,y_0), (1,y_1),& \dots, (n, y_{n}) \big)~| \\ &y_0= i+k, y_{n}=n-i+k, 
            \text{and $\lvert y_{r+1}-y_r \rvert=1$ for $0\leq r \leq n-1$}\}.
        \end{align*}
\end{definition}
    The path description of snake modules of type $A_{2n-1}^{(1)}$ (see Theorem~\ref{path description}) also applies to snake modules of type $A_{n-1}^{(1)}$ \cite[Theorem~6.1]{mukhin2012path}. We note that when $i=0$ or $i=n$, the sets $\mathscr{P}^{A_{n-1}}_{0,k}$ and $\mathscr{P}^{A_{n-1}}_{n,k}$ each consist of only one path, whose associated character is $1$. This gives the path description of the trivial representation.

\begin{lemma}\label{lemma: segment and monomial}
    Given a segment $[l,r]$ of type $A_{n-1}$, let $Y_{r-l,q^{r+l}}$ be its corresponding monomial. A path $p \in \mathscr{P}_{r-l,r+l}$ of type $A_{n-1}$ is a path of the form
    \[p = \big((0,y_0),(1,y_1),\dots,(n,y_n) \big)\]
    such that $y_0 = 2r$ and $y_n = n+ 2l$.
\end{lemma}
\begin{proof}
    This follows immediately from Definition~\ref{def: paths} and the correspondence between segments and monomials as described in Section~\ref{subsec: main theorem}.
\end{proof}

Fix a sequence of points with first coordinate $0$, 
    $$(0,x_0^{(1)}),(0,x_0^{(2)}),\dots,(0,x_0^{(T)})$$
    such that $x_0^{(1)} < x_0^{(2)} < \cdots < x_0^{(T)}$, and a sequence of points with first coordinate $n$,
    $$(n,x_n^{(1)}),(n,x_n^{(2)}),\dots,(n,x_n^{(T)})$$
    such that $x_n^{(1)} < x_n^{(2)} < \cdots < x_n^{(T)}$. Assume that $-n \le x_n^{(t)} - x_0^{(t)} \le n$ for all $1 \le t \le T$.
    
\begin{definition}\label{def: sets A and B}
    Given $M \in \NN$, we define the following sets of NOP of type $A_{n-1}$:

    \[\begin{split}
        &A_{x_0^{(1)},\dots,x_0^{(T)}; x_n^{(1)},\dots,x_n^{(T)};M} := \\
        &\{\overline{p} = (p_1,\dots,p_T)~\text{non-overlapping of type}~A_{n-1}~|~(0,y_0^{(t)}) \in p_t~\text{and}~(n,y_n^{(t)}) \in p_t\\
        &\hspace{4cm} \text{such that}~x_0^{(t-1)} < y_0^{(t)} \le x_0^{(t)}, y_n^{(t)} = x_n^{(t)}~\text{and}~\sum_{t=1}^T(x_0^{(t)} - y_0^{(t)}) = 2M \},
    \end{split}\]
    where $x_0^{(-1)} = -\infty$. In other words, this is the set of paths with fixed right endpoints and left endpoints constrained to a betweenness condition.

    Similarly,
    \[\begin{split}
        &B_{x_0^{(1)},\dots,x_0^{(T)}; x_n^{(1)},\dots,x_n^{(T)};M} := \\
        &\{ \overline{p} = (p_1,\dots,p_T)~\text{non-overlapping of type}~A_{n-1}~|~(0,y_0^{(t)}) \in p_t~\text{and}~(n,y_n^{(t)}) \in p_t\\
        &\hspace{4cm} \text{such that}~y_0^{(t)} = x_0^{(t)}, x_n^{(t)} \le y_n^{(t)} < x_n^{(t+1)}~\text{and}~\sum_{t=1}^T(y_n^{(t)}-x_n^{(t)}) = 2M \},
    \end{split}\]
    where $x_n^{(T+1)} = \infty$.
\end{definition}

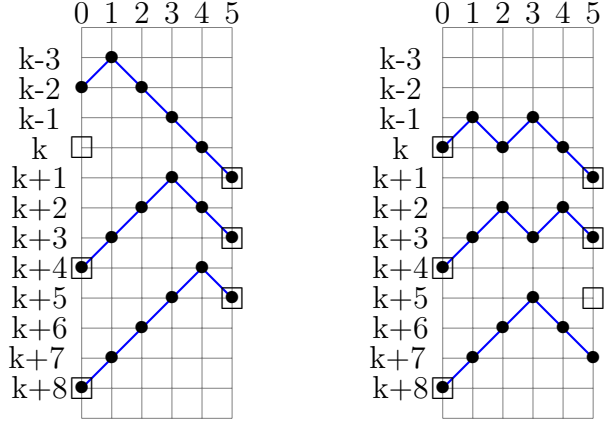
\begin{figure}[htp]
    \centering
\begin{tikzpicture}[scale=0.4]
\draw[step=1cm, gray, very thin] (0, -1) grid (5, 12);
\node at (0, 12.5) {0}; \node at (1, 12.5) {1};
\node at (2, 12.5) {2}; \node at (3, 12.5) {3};
\node at (4, 12.5) {4}; \node at (5, 12.5) {5};

\node at (-1.4, 0) {k+8};
\node at (-1.5, 1) {k+7};
\node at (-1.4, 2) {k+6};
\node at (-1.4, 3) {k+5};
\node at (-1.4, 4) {k+4};
\node at (-1.4, 5) {k+3};
\node at (-1.4, 6) {k+2}; \node at (-1.4, 7) {k+1};
\node at (-1.4, 8) {k}; \node at (-1.4, 9) {k-1};
\node at (-1.4, 10) {k-2}; \node at (-1.4, 11) {k-3};

\draw[blue, thick] (0,0)-- (1,1)-- (2,2)--(3,3)--(4,2)--(5,1);

\draw[blue, thick] (0,4)-- (1,5)-- (2,6)--(3,5)--(4,6)--(5,5);

\draw[blue, thick] (0,8)-- (1,9)-- (2,8)--(3,9)--(4,8)--(5,7);

\node at (0,0) {$\bullet$};
\node at (1,1){$\bullet$}; 
\node at (2,2){$\bullet$}; 
\node at (3,3) {$\bullet$}; 
\node at (4,2) {$\bullet$};
\node at (5,1) {$\bullet$}; 
\node at (0,4) {$\bullet$};
\node at (1,5){$\bullet$}; 
\node at (2,6){$\bullet$}; 
\node at (3,5) {$\bullet$}; 
\node at (4,6) {$\bullet$};
\node at (5,5) {$\bullet$}; 
\node at (0,8) {$\bullet$};
\node at (1,9){$\bullet$}; 
\node at (2,8){$\bullet$}; 
\node at (3,9) {$\bullet$}; 
\node at (4,8) {$\bullet$};
\node at (5,7) {$\bullet$}; 

\node at (5,7) {$\square$};
\node at (5,5) {$\square$}; 
\node at (5,3) {$\square$};
\node at (0,0) {$\square$};
\node at (0,4) {$\square$};
\node at (0,8) {$\square$};

\begin{scope}[xshift=-12cm]
\draw[step=1cm, gray, very thin] (0, -1) grid (5, 12);
\node at (0, 12.5) {0}; \node at (1, 12.5) {1};
\node at (2, 12.5) {2}; \node at (3, 12.5) {3};
\node at (4, 12.5) {4}; \node at (5, 12.5) {5};

\node at (-1.4, 0) {k+8};
\node at (-1.5, 1) {k+7};
\node at (-1.4, 2) {k+6};
\node at (-1.4, 3) {k+5};
\node at (-1.4, 4) {k+4};
\node at (-1.4, 5) {k+3};
\node at (-1.4, 6) {k+2}; \node at (-1.4, 7) {k+1};
\node at (-1.4, 8) {k}; \node at (-1.4, 9) {k-1};
\node at (-1.4, 10) {k-2}; \node at (-1.4, 11) {k-3};

\draw[blue, thick] (0,0)-- (1,1)-- (2,2)--(3,3)--(4,4)--(5,3);

\draw[blue, thick] (0,4)-- (1,5)-- (2,6)--(3,7)--(4,6)--(5,5);

\draw[blue, thick] (0,10)-- (1,11)-- (2,10)--(3,9)--(4,8)--(5,7);

\node at (0,0) {$\bullet$};
\node at (1,1){$\bullet$}; 
\node at (2,2){$\bullet$}; 
\node at (3,3) {$\bullet$}; 
\node at (4,4) {$\bullet$};
\node at (5,3) {$\bullet$}; 
\node at (0,4) {$\bullet$};
\node at (1,5){$\bullet$}; 
\node at (2,6){$\bullet$}; 
\node at (3,7) {$\bullet$}; 
\node at (4,6) {$\bullet$};
\node at (5,5) {$\bullet$}; 
\node at (0,10) {$\bullet$};
\node at (1,11){$\bullet$}; 
\node at (2,10){$\bullet$}; 
\node at (3,9) {$\bullet$}; 
\node at (4,8) {$\bullet$};
\node at (5,7) {$\bullet$};

\node at (5,7) {$\square$};
\node at (5,5) {$\square$}; 
\node at (5,3) {$\square$};
\node at (0,0) {$\square$};
\node at (0,4) {$\square$};
\node at (0,8) {$\square$};
\end{scope}

\end{tikzpicture}
\caption{A triple of NOP in $A_{k,k+4,k+8; k+1,k+3,k+5;1}$ (left) and a triple of NOP in $B_{k,k+4,k+8; k+1,k+3,k+5;1}$ (right). The points $(0,x_0^{(t)})$ and $(n,x_n^{(t)})$ are marked by squares.}
\label{fig: AB correspondence}
\end{figure}
    
\begin{corollary}\label{cor: cerrespondence of type A paths}
    The equality in Theorem~\ref{thm: type A equality} induces a (non-unique) bijection between the sets $$A_{x_0^{(1)},\dots,x_0^{(T)}; x_n^{(1)},\dots,x_n^{(T)};M}~\text{and}~B_{x_0^{(1)},\dots,x_0^{(T)}; x_n^{(1)},\dots,x_n^{(T)};M}.$$ Moreover, if we write $\overline{p}_A$ and $\overline{p}_B$ as the corresponding paths under this bijection, then $$\mathfrak{m}'(\overline{p}_A) = \mathfrak{m}'(\overline{p}_B),$$ where $\mathfrak{m}'$ is the associated $\g$-weight defined in Corollary~\ref{character-path}.
\end{corollary}
\begin{proof}
    By Lemma~\ref{lemma: segment and monomial}, the set of NOP $\overline{p} \in A_{x_0^{(1)},\dots,x_0^{(T)}; x_n^{(1)},\dots,x_n^{(T)};M}$ with left endpoints $(0,y_0^{(1)}),\dots,(0,y_0^{(T)})$ form the path description of the snake module $L([\mathbf{l},\mathbf{r}])$, such that $l_t = \frac{x_n^{(t)} -n}{2}$ and $r_t = \frac{y_0^{(t)}}{2}$, $1 \le t \le T$. Therefore, by Corollary~\ref{character-path}, there is a bijection between the set $A_{x_0^{(1)},\dots,x_0^{(T)}; x_n^{(1)},\dots,x_n^{(T)};M}$ and the set of terms in the character on the left-hand side of \eqref{eqn: type A equation}, which maps a $T$-tuple of NOP $\overline{p}$ to its associated character $\mathfrak{m}'(\overline{p})$.

    Similarly, the set $B_{x_0^{(1)},\dots,x_0^{(T)}; x_n^{(1)},\dots,x_n^{(T)};M}$ is in bijection with terms on the right-hand side of \eqref{eqn: type A equation} in the same way.

    Therefore, Theorem~\ref{thm: type A equality} induces a (non-unique) bijection between paths in the two sets, preserving their associated character.
\end{proof}

In addition to the above corollary, this bijection also preserves what we call the \textit{half-character at $n$} as follows.

\begin{definition}
    For a path $p = \big((0,y_0),(1,y_1),\dots,(n,y_n) \big)$ of type $A_{n-1}$, we say that its half-character at $n$ is $\omega_n'$ if $y_{n} = y_{n-1}-1$, and we say that its half-character at $n$ is $-\omega_n'$ if $y_{n} = y_{n-1}+1$.

    The half-character at $n$ of a $T$-tuple of NOP is defined as the sum of the half-characters at $n$ of each individual path.

    In other words, we have a simple expression that the half-character at $n$ of $\overline{p}$ equals to
    \[\big(\sum_{t=1}^T (y_{n-1}^{(t)} - y_n^{(t)})\big) \omega_n',\]
    where $(n-1,y_{n-1}^{(t)}),(n,y_n^{(t)}) \in p_t$, $1 \le t \le T$.
\end{definition}

\begin{proposition}\label{prop: bijection preseves half-character}
    The bijection in Corollary~\ref{cor: cerrespondence of type A paths} preserves the half-character at $n$.
\end{proposition}
\begin{proof}
    We begin with a single path $p$ with endpoints $(0,y_0)$ and $(n,y_n)$. Let $\mathfrak{m}'(p) = k_1\omega_1 + \cdots k_{n-1}\omega_{n-1}$ be its associated weight. Let $k_n'\omega_n'$ be its half-character at $n$, $k_n' \in \{\pm 1\}$. We claim that 
    \[y_0 - y_n = 2k_1 + 4k_2 + \cdots + 2(n-1)k_{n-1}+nk_n'.\]

    In fact, we note that at each $1 \le j \le n-1$, $(j,y_j)$ is  
    \begin{itemize}
        \item an upper corner of $p$ if and only if $y_{j-1} - 2y_{j} + y_{j+1}= 2$,
        \item a lower corner of $p$ if and only if $y_{j-1} - 2y_{j} + y_{j+1}= -2$,
        \item neither an upper corner nor a lower corner if and only if $y_{j-1} - 2y_{j} + y_{j+1}= 0$.
    \end{itemize}
    Therefore, $2k_j = y_{j-1} - 2y_j + y_{j+1}$. Then we can calculate that 
    \[y_0 - y_n = \sum_{j=1}^{n-1}j(y_{j-1}-2y_j+y_{j+1}) + ny_{n-1} -ny_n = \sum_{j=1}^{n-1} 2jk_j + nk_n'.\]

    Since the weight $\mathfrak{m}'$ of a $T$-tuple of NOP is the sum of the weight of each single path, for a $T$-tuple of NOP $\overline{p} = (p_1,\dots,p_T)$, whose associated weight is $\mathfrak{m}'(\overline{p}) = k_1\omega_1 + \cdots k_{n-1}\omega_{n-1}$, and whose half-character at $n$ is $k_n'\omega_n'$, $k_n' \in \ZZ$, we have 
    \[\sum_{t=1}^T (y_0^{(t)} - y_n^{(t)}) = 2k_1 + 4k_2 + \cdots + 2(n-1)k_{n-1}+nk_n',\]
    where $(0,y_0^{(t)}), (n,y_n^{(t)}) \in p_t$.
    
    Note that the one-to-one correspondence in Corollary~\ref{cor: cerrespondence of type A paths} preserves the weight $\mathfrak{m}'$, and it also preserves the sum $\sum_{t=1}^T (y_0^{(t)} - y_n^{(t)})$, which equals to $\sum_{t=1}^T (x_0^{(t)} - x_n^{(t)}) - 2M$. By the above identity, the one-to-one correspondence preserves $k_n'$ as well.
\end{proof}

\section{Langlands branching rule}\label{sec:Langlands branching rule}
We have proved in Theorem~\ref{thm: langlands dual rep A to B} that for a shortened snake module $V$ of $\Uqghat$ of type $B_n^{(1)}$, there is a shortened snake module ${}^L{V}$ of $\Uqghatdual$, whose highest weights are in correspondence, such that $\chi^{\sigma}({}^L{V}) \preceq \Pi (\chi(V))$. We have also established a criteria, using the notion of gaps, to determine which term of $\Pi(\chi(V))$ appears in the character of ${}^L{V}$. In this section, we establish an equality
\[ \Pi(\chi(V)) = \chi^{\sigma}({}^L{V})  + \sum_W \chi^{\sigma}(W),\]
which decomposes $\Pi(\chi(V))$ into a sum of characters of irreducible representations of $\Uqghatdual$.

\subsection{Paths of non-zero gap}
Given a shortened snake $(i_t,k_t)_{1 \le t \le T}$ of type $B_n$. Let $V = L(\prod_{t:i_t \neq n}Y_{i_t,q^{k_t}} \prod_{t: i_t=n}Y_{n,q^{k_t-1}}Y_{n,q^{k_t+1}})$ be the corresponding shortened snake module of the quantum affine algebra of type $B_n^{(1)}$.

Let $(p_1,\dots,p_T) \in \overline{\mathscr{P}}^B_{(i_t,k_t)_{1\leq t \leq T}}$ be a $T$-tuple of NOP for this module. Recall that we have defined a map $F^{-1}$ in Proposition~\ref{prop: level 0 bijection} which associates each path $p \in \mathscr{P}_{i,k}^B$ of gap $0$ to a path $F^{-1}(p)$ in $\mathscr{P}^A_{i,\frac{k}{2}}$, where the notion gap is defined in Definition~\ref{def: level}.

In this subsection, we generalize the map $F^{-1}$ to paths of arbitrary gap.

For a single path $p \in \mathscr{P}_{i,k}^B$, suppose that
    \begin{align*}
        && p = \big( (0,y_0), &(2,y_1), \dots, (2n-4, y_{n-2}), (2n-2,y_{n-1}), (2n-1,y_n), \\
        && &(2n-1,z_n), (2n,z_{n-1}),\dots,(4n-4,z_1),(4n-2,z_0) \big),
    \end{align*}
then we define a pair of type $A_{n-1}$ paths by
    \begin{align*}
        L(p) = \big( (0,\frac{y_0}{2}),(1,\frac{y_1}{2}), \dots, (n-1, \frac{y_{n-1}}{2}), (n,\lfloor \frac{y_{n}+1}{2} \rfloor) \big) \in \mathscr{P}^{A_{n-1}}_{i,k},
    \end{align*}
    where $i = \frac{y_0-2 \lfloor \frac{y_{n}+1}{2} \rfloor+ 2n}{4}$, $k = \frac{y_0+ 2 \lfloor \frac{y_{n}+1}{2} \rfloor-2n}{4}$,
    and 
    \begin{align*}
        R(p) = \big((n,\lfloor \frac{z_n + 3}{2} \rfloor), (n+1,\frac{z_{n-1}+2}{2}),\dots,(2n-1,\frac{z_1+2}{2}), (2n,\frac{z_0+2}{2}) \big),
    \end{align*}
     where we label the first coordinate from $n$ to $2n$, instead of from $0$ to $n$, for later simplicity.
     
$L(p)$ and $R(p)$ are respectively determined by the left and the right branch of $p$.

For a $T$-tuple of NOP $\overline{p} = (p_1,\dots,p_T) \in \overline{\mathscr{P}}^B_{(i_t,k_t)_{1\leq t \leq T}}$, we associate it with a pair of NOP of type $A_{n-1}$: $(L(p_1),\dots,L(p_T))$ and $(R(p_1),\dots,R(p_T))$.

\begin{lemma}
    We have
    \[\lfloor \frac{z_n^{(t)}+3}{2} \rfloor \le \lfloor \frac{y_n^{(t)}+1}{2} \rfloor, \quad \forall 1 \le t \le T,\]
    and 
    \[ \lfloor \frac{y_n^{(t)}+1}{2} \rfloor < \lfloor \frac{z_n^{(t+1)}+3}{2} \rfloor, \quad \forall 1 \le t \le T-1. \]
\end{lemma}
\begin{proof}
    We have seen the first equality in the definition of gaps. For the second inequality, one notice that $p_t \succ p_{t+1}$ implies $y_n^{(t)} < z_n^{(t+1)}$, thus $\lfloor \frac{y_n^{(t)}+1}{2} \rfloor < \lfloor \frac{z_n^{(t+1)}+3}{2} \rfloor$.
\end{proof}

The following proposition is an analogue of Lemma~\ref{lemma: F injective} and Proposition~\ref{prop: level 0 bijection}.
\begin{proposition}\label{prop: bijection type B and pair of type A}
    The map which associates $\overline{p}$ with $(L(\overline{p}),R(\overline{p}))$ defines a bijection between the set $\overline{\mathscr{P}}^B_{(i_t,k_t)_{1\leq t \leq T}}$ and the set
    \[\begin{split}
        \{(\overline{q}',\overline{q})~&\text{a pair of $T$-tuples of NOP of type}~A_{n-1}~|~(0,i_t+\frac{k_t}{2}) \in q_t',\\
        &(2n,2n-i_t+\frac{k_t}{2}) \in q_t,~{x_n}^{(1)} \le {x_n'}^{(1)} < {x_n}^{(2)} \le {x_n'}^{(2)} < \cdots < {x_n}^{(T)} \le {x_n'}^{(T)}\},
    \end{split}
    \]
    where $\overline{q}' = (q_1',\dots,q_T')$, $\overline{q} = (q_1,\dots,q_T)$, and $(n,{x_n'}^{(t)}) \in q_t'$, $(n,{x_n}^{(t)}) \in q_t$, $\forall t$.
\end{proposition}
\begin{proof}
    The injectivity follows immediately from the definition of $L(p)$ and $R(p)$.

    We prove the surjectivity by writing down the preimage explicitly. Given a pair of paths $(q',q)$ of type $A_{n-1}$ such that $(0,i+\frac{k}{2}) \in q'$, $(2n,2n-i+\frac{k}{2}) \in q$ and $x_n \le x_n'$, where $(n,x_n') \in q'$, $(n,x_n) \in q$, we write their points as $(j,x_j') \in q'$ and $(j,x_j) \in q$, $0 \le j \le n$.
    
    Its preimage is the type $B_n$ path in $\mathscr{P}^B_{i,k}$ of the form 
    \begin{align*}
        p = \big( (0,y_0), &(2,y_1), \dots, (2n-4, y_{n-2}), (2n-2,y_{n-1}), (2n-1,y_n), \\
         &(2n-1,z_n), (2n,z_{n-1}),\dots,(4n-4,z_1),(4n-2,z_0) \big),
    \end{align*}
    with coordinates 
    \[y_j = 
    \begin{cases}
        2x_j', &\text{if}~j < n,\\
        2x_n'-1 + \epsilon, &\text{if}~j=n~\text{and}~x_n'>x_{n-1}', \\
        2x_n'+1-\epsilon, &\text{if}~j=n~\text{and}~x_n'<x_{n-1}',
    \end{cases}
    \]
    and
    \[z_j = 
    \begin{cases}
        2x_{2n-j}-2, &\text{if}~j < n,\\
        2x_n-3 + \epsilon, &\text{if}~j=n~\text{and}~x_n>x_{n+1}, \\
        2x_n-1-\epsilon, &\text{if}~j=n~\text{and}~x_n<x_{n+1}.
    \end{cases}
    \]
    One verifies directly that $L(p) = q'$ and $R(p) = q$.

    For a pair of $T$-tuples of NOP $(\overline{q}',\overline{q})$ satisfying the condition in the proposition, the preimage is the $T$-tuple of NOP formed by taking the preimage of each component. Note that, by definition of type $A_{n-1}$ paths, we have ${x_n'}^{(t)} \equiv {x_n}^{(t+1)} \pmod 2$. Therefore, the condition ${x_n'}^{(t)} < {x_n}^{(t+1)}$ implies ${x_n'}^{(t)} \le {x_n}^{(t+1)} -2$, which ensures that the preimage is non-overlapping.
\end{proof}

\begin{remark}
    We notice that the formula for the preimage of $(L(p),R(p))$ is almost identical to the formula of the map $F$ defined in Definition~\ref{def: map F between paths}, except that $F^{-1}$ requires the path to have gap $0$. Therefore, this construction can be regarded as a generalization of the map $F$ to the case of arbitrary gaps.
\end{remark}

Recall that the gap of $\overline{p} = (p_1,\dots,p_T) \in \overline{\mathscr{P}}^B_{(i_t,k_t)_{1\leq t \leq T}}$ is defined as \[\mathrm{gap}(\overline{p}) := \sum_{t=1}^T (\lfloor \frac{y_n^{(t)}+1}{2} \rfloor -  \lfloor \frac{z_n^{(t)}+3}{2} \rfloor),\]
where $(2n-1,y_n^{(t)}), (2n-1,z_n^{(t)}) \in p_t$. This suggests that $\overline{p}$ has gap $0$ if and only if the right endpoints of $L(p_t)$ coincide with the left endpoints of $R(p_t)$ for all $1 \le t \le T$.

We have proved in Section~\ref{subsec: Langlands dual from B to A} that $T$-tuples of NOP $(p_1,\dots,p_T) \in \overline{\mathscr{P}}^B_{(i_t,k_t)_{1\leq t \leq T}}$ having gap $0$ are in one-to-one correspondence with $T$-tuples of NOP in $\overline{\mathscr{P}}^A_{(i_t,\frac{k_t}{2})_{1\leq t \leq T}}$. In this section, we construct a correspondence between all NOP $(p_1,\dots,p_T) \in \overline{\mathscr{P}}^B_{(i_t,k_t)_{1\leq t \leq T}}$ and type $A_{2n-1}$ NOP.

\subsection{Bijection between type \texorpdfstring{$B_n$}{} paths and type \texorpdfstring{$A_{2n-1}$}{} paths}
Fix a shortened snake $(i_t,k_t)_{1 \le t \le T}$ of type $B_n$. Fix a $T$-tuple of NOP $\overline{q} = (q_1,\dots,q_T)$ of type $A_{n-1}$ such that $(2n,2n-i_t+\frac{k_t}{2}) \in q_t$. Here we label their first coordinates of $\overline{q}$ from $n$ to $2n$, instead of from $0$ to $n$, as above. We denote the left endpoints of $q_t$ by $(n,x^{(t)})$, $1 \le t \le T$.

\begin{definition}
    For a given $T$-tuple of NOP $\overline{q}$ of type $A_{n-1}$, we denote by
    \[L_{\overline{q}} = \{(p_1,\dots,p_T) \in \overline{\mathscr{P}}^B_{(i_t,k_t)_{1\leq t \leq T}}~|~(R(p_1),\dots,R(p_T)) = \overline{q} \}.\]
\end{definition}

In other words, this is the set of $(p_1,\dots,p_T) \in \overline{\mathscr{P}}^B_{(i_t,k_t)_{1\leq t \leq T}}$ such that their associated $(R(p_1),\dots,R(p_T))$ are the same and equal to the given $\overline{q}$.

By the path description of type $B_n$ snake modules, $(p_1,p_2,\dots,p_T) \in L_{\overline{q}}$ if and only if $(R(p_1),\dots,R(p_T))=\overline{q}$ and $(L(p_1),\dots,L(p_T))$ is a $T$-tuple of NOP of type $A_{n-1}$ such that 
\begin{itemize}
    \item the left endpoints are fixed by the snake as $(0,\frac{y_0^{(t)}}{2}) = (0,\frac{2i_t+k_t}{2}) \in L(p_t)$,
    \item and the right endpoints are $(n,a_n^{(t)}): = (n, \lfloor \frac{y_{n}^{(t)}+1}{2} \rfloor)$ such that 
    \[x^{(t)} \le a_n^{(t)} < x^{(t+1)}, \]
    where $x^{(t)} = \lfloor \frac{z_n^{(t)} + 3}{2} \rfloor$.
\end{itemize}

In other words, $(L(p_1),\dots,L(p_T))$ is a $T$-tuple of NOP of type $A_{n-1}$ in the set $$B_{\frac{2i_1+k_1}{2},\dots,\frac{2i_T+k_T}{2};x^{(1)},\dots,x^{(T)};M}$$
defined in Definition~\ref{def: sets A and B}, for some integer $M \in \NN$. Here, by definition, $M$ equals to the gap of $\overline{p} = (p_1,\dots,p_T)$.

By Corollary~\ref{cor: cerrespondence of type A paths}, these $T$-tuples of NOP are in one-to-one correspondence with $T$-tuples of NOP in the set 
$A_{\frac{2i_1+k_1}{2},\dots,\frac{2i_T+k_T}{2};x^{(1)},\dots,x^{(T)};M}$, in such a way that the associated character $\mathfrak{m}'(L(p_1),\dots,L(p_T))$ and also the associated half-character at $n$ are preserved.

Notice that for any $T$-tuple of NOP $(p_1,\dots,p_T)$ in $A_{\frac{2i_1+k_1}{2},\dots,\frac{2i_T+k_T}{2};x^{(1)},\dots,x^{(T)};M}$, the right endpoint of $p_t$ coincides with the left endpoint of $q_t$, which are both $x^{(t)} = \lfloor \frac{z_n^{(t)}+3}{2} \rfloor$. Therefore, these paths together with $\overline{q}$ connected to the right from a $T$-tuple of NOP of type $A_{2n-1}$.

In this way, we have associated a $T$-tuple of NOP of type $B_{n}$ with a $T$-tuple of NOP of type $A_{2n-1}$. An example is illustrated in Figure~\ref{fig: s and p in correspondence}.

In conclusion, we have
\begin{proposition}\label{prop: decomposition theorem}
    Given a fixed $T$-tuple of NOP $\overline{q} = (q_1,\cdots,q_T)$ of type $A_{n-1}$ and a fixed integer $M \in \NN$, let $(n,x^{(t)}) \in q_t$ be the left endpoint of $q_t$. There is a bijection between the set
    $$\{\overline{p} \in L_{\overline{q}}~|~\mathrm{gap}(\overline{p}) = M\}$$
    and the set 
    \begin{align*}
        \mathcal{S}_{\overline{q},M} : = \{\overline{s} = (s_1,\dots,s_T)~\text{$T$-tuple of NOP of type}~A_{2n-1}~| \\
        \overline{s}|_{[0,n]} \in A_{\frac{2i_1+k_1}{2},\dots,\frac{2i_T+k_T}{2};x^{(1)},\dots,x^{(T)};M}~\text{and}~\overline{s}|_{[n,2n]} = \overline{q} \},
    \end{align*}
    where $\overline{s}|_{[0,n]}$ is the $T$-tuple of NOP of type $A_{n-1}$ formed by the subsequences of each $s_t$ consisting of points whose first coordinate ranges from $0$ to $n$. The tuple $\overline{s}|_{[n,2n]}$ is defined similarly. 
    
    Moreover, this correspondence preserves character in the sense that
    \[\Pi \circ \mathfrak{m}'(\overline{p}) = \varpi \circ \mathfrak{m}'(\overline{s})\]
    for all $\overline{p}$ and $\overline{s}$ in correspondence, where $\varpi$ is the folding map in \eqref{eqn: folding character}.
\end{proposition}
\begin{proof}
    The bijection has been explained above. We verify the equation of associated weights. 
    
    Recall that for a tuple of NOP $\overline{p}$ of type $B_n$, its associated character \begin{equation}\label{eqn: mp}
    \mathfrak{m}'(\overline{p}) = \mathfrak{m}'(L(\overline{p})) + \mathfrak{m}'(R(\overline{p})) + k_n'(L(\overline{p})) + k_n'(R(\overline{p})),
    \end{equation}
    where $\mathfrak{m}'(L(\overline{p}))$ and $\mathfrak{m}'(R(\overline{p}))$ are the associated defined by upper corners and lower corners, which is a linear combination of $\omega_j$ ($1 \le j \le n-1$), and $k_n'(L(\overline{p}))$ (resp. $k_n'(R(\overline{p}))$) is the half-character at $n$ of $L(\overline{p})$ (resp. $R(\overline{p})$), which is a multiple of $\omega_n$.

    By Proposition~\ref{prop: bijection preseves half-character}, the bijection which maps $L(\overline{p})$ to $\overline{s}|_{[0,n]}$ preserves the associated character and the half-character at $n$. Therefore, $\mathfrak{m}'(L(\overline{p})) = \mathfrak{m}'(\overline{s}|_{[0,n]})$ and
    \[k_n'(L(\overline{p})) = k_n'(\overline{s}|_{[0,n]}) = \sum_{t=1}^T (b_{n-1}^{(t)} - b_n^{(t)}),\]
    where $(n-1,b_{n-1}^{(t)}), (n,b_n^{(t)}) \in s_t$, $1 \le t \le T$.

    The right branch $R(\overline{p})$ coincides with $\overline{s}|_{[n,2n]}$, thus 
    \[k_n'(R(\overline{p})) = k_n'(\overline{s}|_{[n,2n]}) = \sum_{t=1}^T (b_{n+1}^{(t)} - b_n^{(t)}),\]
    where $(n+1,b_{n+1}^{(t)})\in s_t$, $1 \le t \le T$.

    Recall that the multiplicity of $\omega_n$ in $\mathfrak{m}'(\overline{s})$ equals to
    \[\frac{1}{2} \sum_{t=1}^T (b_{n-1}^{(t)} - 2b_n^{(t)} + b_{n+1}^{(t)}),\]
    as we have seen in the proof of Proposition~\ref{prop: bijection preseves half-character}. Thus we have 
    \begin{equation}\label{eqn: ms}
        \mathfrak{m}'(\overline{s}) = \mathfrak{m}'(\overline{s}|_{[0,n]}) + \mathfrak{m}'(\overline{s}|_{[n,2n]}) + \frac{1}{2} (k_n'(L(\overline{p})) + k_n'(R(\overline{p})) ).
    \end{equation}
    
    The identity follows from \eqref{eqn: mp}, \eqref{eqn: ms} and the definition of $\varpi$ and $\Pi$.
\end{proof}

\begin{figure}[htp]
    \centering
\begin{tikzpicture}[scale=0.4]
\draw[step=1cm, gray, very thin] (0, -1) grid (10, 12);
\node at (0, 12.5) {0}; \node at (1, 12.5) {1};
\node at (2, 12.5) {2}; \node at (3, 12.5) {3};
\node at (4, 12.5) {4}; \node at (5, 12.5) {5}; \node at (6, 12.5) {6}; \node at (7, 12.5) {7}; \node at (8, 12.5) {8}; \node at (9, 12.5) {9}; \node at (10, 12.5) {10};

\node at (-1.4, 0) {k+8};
\node at (-1.5, 1) {k+7};
\node at (-1.4, 2) {k+6};
\node at (-1.4, 3) {k+5};
\node at (-1.4, 4) {k+4};
\node at (-1.4, 5) {k+3};
\node at (-1.4, 6) {k+2}; \node at (-1.4, 7) {k+1};
\node at (-1.4, 8) {k}; \node at (-1.4, 9) {k-1};
\node at (-1.4, 10) {k-2}; \node at (-1.4, 11) {k-3};

\draw[blue, thick] (0,0)-- (1,1)-- (2,2)--(3,3)--(4,4)--(5,3)--(6,4)--(7,3)--(8,2)--(9,1)--(10,0);

\draw[blue, thick] (0,4)-- (1,5)-- (2,6)--(3,7)--(4,6)--(5,5)--(6,6)--(7,5)--(8,4)--(9,3)--(10,2);

\draw[blue, thick] (0,10)-- (1,11)-- (2,10)--(3,9)--(4,8)--(5,7)--(6,8)--(7,9)--(8,8)--(9,7)--(10,6);

\node at (0,0) {$\bullet$};
\node at (1,1){$\bullet$}; 
\node at (2,2){$\bullet$}; 
\node at (3,3) {$\bullet$}; 
\node at (4,4) {$\bullet$};
\node at (5,3) {$\bullet$}; \node at (6,4) {$\bullet$}; \node at (7,3) {$\bullet$}; \node at (8,2) {$\bullet$}; \node at (9,1) {$\bullet$}; \node at (10,0) {$\bullet$}; 
\node at (0,4) {$\bullet$};\node at (1,5){$\bullet$}; 
\node at (2,6){$\bullet$}; 
\node at (3,7) {$\bullet$}; 
\node at (4,6) {$\bullet$};
\node at (5,5) {$\bullet$}; \node at (6,6) {$\bullet$}; \node at (7,5) {$\bullet$}; \node at (8,4) {$\bullet$}; \node at (9,3) {$\bullet$}; \node at (10,2) {$\bullet$}; 
\node at (0,10) {$\bullet$};
\node at (1,11){$\bullet$}; 
\node at (2,10){$\bullet$}; 
\node at (3,9) {$\bullet$}; 
\node at (4,8) {$\bullet$};
\node at (5,7) {$\bullet$}; \node at (6,8) {$\bullet$}; \node at (7,9) {$\bullet$}; \node at (8,8) {$\bullet$}; \node at (9,7) {$\bullet$}; \node at (10,6) {$\bullet$}; 

\begin{scope}[xshift=15cm]
\draw[step=1cm, gray, very thin] (0, -1) grid (18, 21);
\node at (0, 21.5) {0}; \node at (2, 21.5) {2};
\node at (4, 21.5) {4}; \node at (6, 21.5) {6};
\node at (8, 21.5) {8};\node at (9, 21.5) {9}; 
\node at (10, 21.5) {10}; \node at (12, 21.5) {12};
\node at (14, 21.5) {14};\node at (16, 21.5) {16}; \node at (18, 21.5) {18};

\node at (0, -1.5) {0}; \node at (2, -1.5) {1};
\node at (4, -1.5) {2}; \node at (6, -1.5) {3};
\node at (8, -1.5) {4}; \node at (9, -1.5) {5};
\node at (10, -1.5) {4}; \node at (12, -1.5) {3};
\node at (14, -1.5) {2};\node at (16, -1.5) {1};\node at (18, -1.5) {0};

\node at (-1.7, 0) {2k+18};
\node at (-1.7, 1) {2k+17};
\node at (-1.7, 2) {2k+16};
\node at (-1.7, 3) {2k+15};
\node at (-1.7, 4) {2k+14};
\node at (-1.7, 5) {2k+13};
\node at (-1.7, 6) {2k+12}; \node at (-1.7, 7) {2k+11};
\node at (-1.7, 8) {2k+10}; \node at (-1.55, 9) {2k+9};
\node at (-1.55, 10) {2k+8};\node at (-1.55, 11) {2k+7};\node at (-1.55, 12) {2k+6}; \node at (-1.55, 13) {2k+5}; \node at (-1.55, 14) {2k+4}; \node at (-1.55, 15) {2k+3}; \node at (-1.55, 16) {2k+2}; \node at (-1.55, 17) {2k+1}; \node at (-1.55, 18) {2k}; \node at (-1.55, 19) {2k-1}; \node at (-1.55, 20) {2k-2};

\draw[blue, thick] (0,16)--(2,18)--(4,16)--(6,18)--(8,16)--(9,14.7);
\draw[blue, thick] (9,16.7)--(10,18)--(12,20)--(18,14);

\node at (0,16) {$\bullet$}; 
\node at (2,18) {$\bullet$};
\node at (4,16) {$\bullet$};
\node at (6,18) {$\bullet$}; 
\node at (8,16) {$\bullet$};
\node at (9,14.7) {$\bullet$};
\node at (9,16.7) {$\bullet$};
\node at (10,18) {$\bullet$};
\node at (12,20) {$\bullet$};\node at (14,18) {$\bullet$};\node at (16,16) {$\bullet$}; \node at (18,14) {$\bullet$};  

\draw[blue, thick] (0,8)--(2,10)--(4,12)--(6,10)--(8,12)--(9,10.7);
\draw[blue, thick] (9,12.7)--(10,14)--(12,12)--(18,6);

\node at (0,8) {$\bullet$}; 
\node at (2,10) {$\bullet$};
\node at (4,12) {$\bullet$};
\node at (6,10) {$\bullet$}; 
\node at (8,12) {$\bullet$};
\node at (9,10.7) {$\bullet$};
\node at (9,12.7) {$\bullet$};
\node at (10,14) {$\bullet$};
\node at (12,12) {$\bullet$};\node at (14,10) {$\bullet$};\node at (16,8) {$\bullet$}; \node at (18,6) {$\bullet$};

\draw[blue, thick] (0,0)--(2,2)--(4,4)--(6,6)--(8,4)--(9,2.7);
\draw[blue, thick] (9,8.7)--(10,10)--(18,2);

\node at (0,0) {$\bullet$}; 
\node at (2,2) {$\bullet$};
\node at (4,4) {$\bullet$};
\node at (6,6) {$\bullet$}; 
\node at (8,4) {$\bullet$};
\node at (9,2.7) {$\bullet$};
\node at (9,8.7) {$\bullet$};
\node at (10,10) {$\bullet$};
\node at (12,8) {$\bullet$};\node at (14,6) {$\bullet$};\node at (16,4) {$\bullet$}; \node at (18,2) {$\bullet$};

\end{scope}

\end{tikzpicture}
\caption{Two triples of NOP $\overline{s}$ (left) and $\overline{p}$ (right) in correspondence. Note that $\overline{s}|_{[0,n]}$ corresponds to $L(\overline{p})$ as in Figure~\ref{fig: AB correspondence} and $\overline{s}|_{[n,2n]} = R(\overline{p})$.}
\label{fig: s and p in correspondence}
\end{figure}
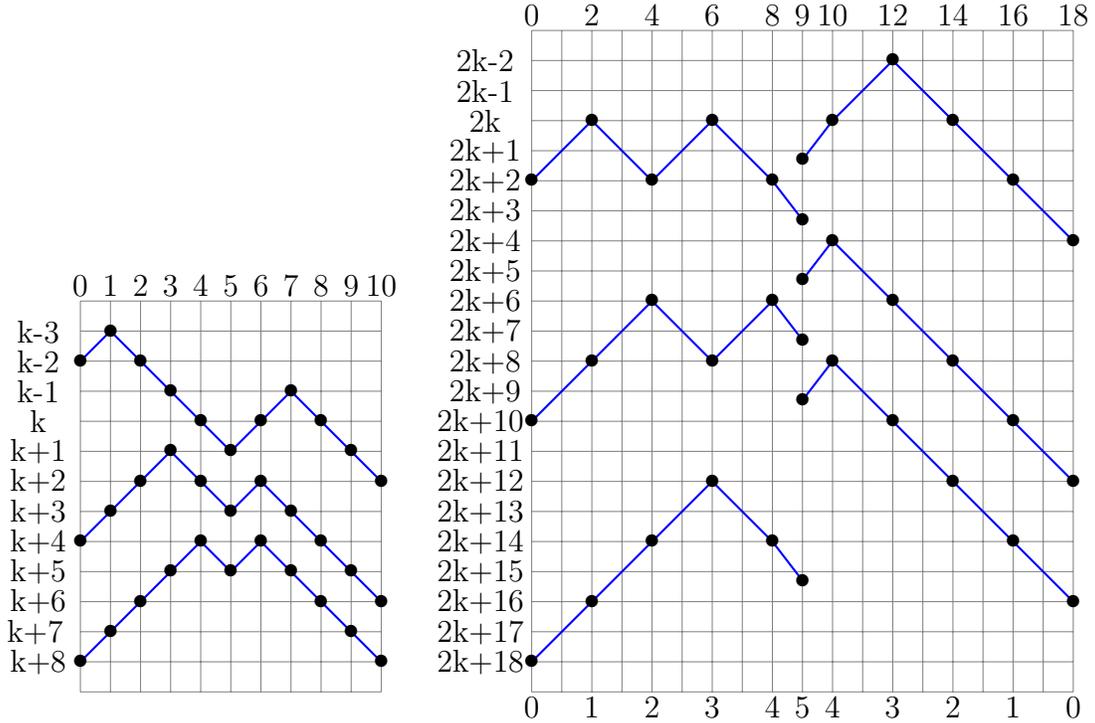

Combining all possible $\overline{q} = (R(p_1),\dots,R(p_T))$ and all $M \in \NN$, the following theorem is derived immediately from the the above proposition.

\begin{theorem}\label{thm: bijection PB and S}
    Given a shortened snake $(i_t,k_t)_{1 \le t \le T}$ of type $B_n$, there is a bijection between the set $\overline{\mathscr{P}}^B_{(i_t,k_t)_{1 \le t \le T}}$ and the set 
     \begin{equation}\label{eq: set S}
        \begin{split}
        \mathcal{S} := \{\overline{s} = (s_1,\dots,s_T)~\text{NOP of type}~A_{2n-1}~|~(2n,2n-i_t+\frac{k_t}{2}) \in s_t \\
        \text{and}~(0,y_0^{(t)})\in s_t~\text{with}~i_{t-1}+\frac{k_{t-1}}{2} < y_0^{(t)} \le  i_t + \frac{k_t}{2} , \forall 1 \le t \le T \},
        \end{split}
    \end{equation}
    where we view $i_0 + \frac{k_0}{2}$ as $-\infty$.

    Moreover, this bijection preserves the character in the sense that
    \[\Pi \circ \mathfrak{m}'(\overline{p}) = \varpi \circ \mathfrak{m}'(\overline{s})\]
    for all $\overline{p} \in \overline{\mathscr{P}}^B_{(i_t,k_t)_{1 \le t \le T}}$ and $\overline{s} \in \mathcal{S}$ in correspondence.
\end{theorem}
\begin{proof}
    Notice that
    \[\overline{\mathscr{P}}^B_{(i_t,k_t)_{1 \le t \le T}} = \bigsqcup_{\substack{\overline{q}~\text{such that}\\(2n,2n-i_t+\frac{k_t}{2}) \in q_t}} \bigsqcup_{M \in \NN} \{\overline{p} \in L_{\overline{q}}~|~\mathrm{gap}(\overline{p}) = M\},\]
    and 
    \begin{align*}
        \mathcal{S} = \bigsqcup_{\substack{\overline{q}~\text{such that}\\(2n,2n-i_t+\frac{k_t}{2}) \in q_t}} \bigsqcup_{M \in \NN} \mathcal{S}_{\overline{q},M}.
    \end{align*}
    The theorem follows from Proposition~\ref{prop: decomposition theorem} immediately.
\end{proof}

This theorem states that type $B_n$ paths are in $\overline{\mathscr{P}}^B_{(i_t,k_t)_{1 \le t \le T}}$ is in bijection with type $A_{2n-1}$ paths whose endpoints satisfy certain betweenness conditions, and the bijection preserves the corresponding weights. By applying Corollary~\ref{character-path}, these weights form characters of snake modules. Thus we can conclude our main theorem.

\begin{theorem}\label{thm: langlands branching rule}
    Let $(i_t,k_t)_{1 \le t \le T}$ be a shortened snake of type $B_n$. Let 
    $$m = \prod_{t:i_t \neq n}Y_{i_t,q^{k_t}} \prod_{t: i_t=n}Y_{n,q^{k_t-1}}Y_{n,q^{k_t+1}}$$ be the dominant monomial defined by the snake. Then
    \begin{equation}\label{eqn: decomposition equality}
        \Pi(\chi(L(m))) = \sum_{s_t} \chi^{\sigma}(L(\prod_{t=1}^T Z_{i_t - s_t,q^{\frac{k_t}{2}- s_t}})),
    \end{equation}
    where the sum on the right-hand side is taken over $T$-tuples of integers $(s_1,\dots,s_T)$ such that $0 \le s_t < \min( \frac{2i_t+k_t-2i_{t-1}-k_{t-1}}{4}, i_t+1)$, where a term $Y_{0,q^c}$ ($c \in \ZZ$) should be understood as $1$.

    In particular, the leading term, which corresponds to the $T$-tuple of integers $(s_1,\dots,s_T) = (0,\dots,0)$ is the Langlands dual representation that we obtained in Theorem~\ref{thm: langlands dual rep A to B}.

    Equivalently, we can rewrite \eqref{eqn: decomposition equality} in terms of multisegments
    \begin{equation}
    \Pi(\chi(L(m))) = \sum_{[\mathbf{l},\mathbf{r}]} \varpi (\chi(L([\mathbf{l}, \mathbf{r}]))),
    \end{equation}
    where the sum on the right-hand side is taken over all snake modules of type $A_{2n-1}$ such that $l_t = \frac{-2i_t+k_t+2n+2}{4}$ and $\frac{2i_{t-1}+k_{t-1}+2n+2}{4} < r_t \le \frac{2i_t+k_t+2n+2}{4}$, $\forall 1 \le t \le T$. Here $\frac{-2i_0+k_0+2n+2}{4}$ is understood as $-\infty$. In the formula, we shift all segments by $\frac{2n+2}{4}$ to ensure that $l_t,r_t$ are integers.

\end{theorem}
\begin{proof}
    Depending on the left endpoints of paths in \eqref{eq: set S}, we can decompose the set $\mathcal{S}$ into a disjoint union of type $A_{2n-1}$ NOP
    \[
    \mathcal{S} = \bigsqcup_{(u_t,v_t)_{1 \le t \le T}} \overline{\mathscr{P}}^A_{(u_t,v_t)_{1 \le t \le T}},
    \]
    where the index runs over type $A_{2n-1}$ snakes $(u_t,v_t)_{1 \le t \le T}$ such that 
    \[i_{t-1} + \frac{k_{t-1}}{2} < u_t + v_t \le i_t + \frac{k_t}{2}, \; \text{and} \; 2n-u_t+v_t = 2n - i_t + \frac{k_t}{2}. \]
    Equivalently, 
    \[\frac{2i_{t-1} +k_{t-1} +2i_t - k_t}{4} < u_t \le i_t, \; \text{and} \; v_t = u_t - i_t + \frac{k_t}{2}\]
    Here we use a slight abuse of notation in that the sequence $(u_t,v_t)_{1 \le t \le T}$ is not forced to have length $T$. A term $(u_t,v_t)$ should be omitted if $u_t  = 0$, and the remaining terms form a snake of length less than $T$.

    By the bijection constructed in Theorem~\ref{thm: bijection PB and S} between the set $\overline{\mathscr{P}}^B_{(i_t,k_t)_{1 \le t \le T}}$ and the set $\mathcal{S}$, we obtain a bijection between type $B_n$ NOP and the disjoint union of type $A_{2n-1}$ NOP. Hence we obtain the desired equality between characters.
\end{proof}

We close this section with two examples of the branching rule.

\begin{example}
    For a fixed $i \in I$, let $(i_t,k_t) = (i,4t)$, $1 \le t \le T$, be a shortened snake of type $B_n$. The associated snake module is 
    \[W^{(i)}_{T,q^4} = L(Y_{i,q^4}Y_{i,q^8}Y_{i,q^{12}}\cdots Y_{i,q^{4T}})~\text{if}~i \neq n,\]
    \[W^{(n)}_{2T,q^3} = L(Y_{n,q^3}Y_{n,q^5}Y_{n,q^7}\cdots Y_{i,q^{4T+1}})~\text{if}~i = n,\]
    which are known as Kirillov-Reshetikhin modules.

    Then the Langlands branching rule states that
    \[\Pi(\chi(W^{(i)}_{T,q^4})) = \sum_{s=0}^{i} \chi^{\sigma}(L(Z_{i-s,q^{2-s}}Z_{i,q^4}Z_{i,q^6}\cdots Z_{i,q^{2T}})),\]
    where $Z_{0,q^{2-i}}$ in the last term is understood as $1$.

    Similarly,
    \[\Pi(\chi(W^{(n)}_{2T,q^3})) = \sum_{s=0}^{n} \chi^{\sigma}(L(Z_{n-s,q^{2-s}}Z_{n,q^4}Z_{n,q^6}\cdots Z_{n,q^{2T}})),\]
    where $Z_{0,q^{2-n}}$ in the last term is understood as $1$.
\end{example}

It is possible that two non-isomorphic modules have the same usual character. In this case, the branching rule above will provide two different decompositions of the same character.
\begin{example}
    Let $\lambda = \sum_{i=1}^n \lambda_i \omega_i \in P^+ \cap P'$. Consider the two minimal affinizations of $V(\lambda)$ of type $B_n^{(1)}$ \cite{chari1995minimalnonsimp}, $L(m)$ and $L(m')$, where 
    \[m = \prod_{i=1}^{n-1} \prod_{k=1}^{\lambda_i} Y_{i,q^{4(\lambda_1 + \cdots + \lambda_{i-1}) + 2(i-1) +4k}} \times \prod_{l=1}^{\lambda_n} Y_{n, q^{4(\lambda_1 + \cdots + \lambda_{n-1}) + 2(n-1) +1 + 2l} },\]
    and
    \[m' = \prod_{i=1}^{n-1} \prod_{k=1}^{\lambda_i} Y_{i,q^{-4(\lambda_1 + \cdots + \lambda_{i-1}) - 2(i-1) -4k}} \times \prod_{l=1}^{\lambda_n} Y_{n, q^{-4(\lambda_1 + \cdots + \lambda_{n-1}) - 2(n-1) -1 - 2l} }.\]
    It is known that $\chi(L(m)) = \chi(L(m'))$.

    To simplify the formula, we will use $a_i :  = \lambda_1 + \cdots + \lambda_{i-1}$ in the following.

    Note that the snake corresponding to $m$ starts with $(i_1,k_1) = (1,4)$ and ends with $(i_T,k_T) = ( n,4a_{n} + 2(n-1) +1 + 2\lambda_n )$, while the snake corresponding to $m'$ starts with $(i_1,k_1) = (n, -4a_{n} - 2(n-1) -1 - 2\lambda_n )$ and ends with $(1,-4)$.
    
    The branching rule of $\chi(L(m))$ reads
    \[\Pi(\chi(L(m))) = \sum_{\lambda'} \chi^{\sigma}(L(\prod_{i=1}^{n} \prod_{k=1}^{\lambda'_i} Z_{i,q^{2\delta + 2(\lambda'_1 + \cdots + \lambda'_{i-1}) + (i-1) +2k}})), \]
    where the sum is taken over non-negative integers $(\lambda'_1,\dots,\lambda'_n)$ such that 
    \[\sum_{k=j}^{n} \lambda_j' \in \{ \sum_{k=j}^{n-1} \lambda_j + \frac{\lambda_n}{2}, \sum_{k=j}^{n-1} \lambda_j + \frac{\lambda_n}{2} - 1 \},\] 
    for all $1 \le j \le n$. Here we have a universal shift parameter $\delta : = \sum_{k=1}^{n-1} \lambda_j + \frac{\lambda_n}{2} - \sum_{k=1}^{n} \lambda_j'$ in the formula, which in fact has no impact on the character.

    Note that each summand in the decomposition is the folding of a minimal affinization of type $A_{2n-1}^{(1)}$, which has the same character as the classical Lie algebra representation as evaluation representations. Therefore, we can rewrite the formula as 
    \[\Pi(\chi(L(m))) = \sum_{\lambda'} \varpi \circ \chi(V(\lambda')),\]
    where $V(\lambda')$ is the classical Lie algebra $\mathfrak{sl}_{2n}$-module with highest weight $\lambda' = \lambda_1' \omega_1 + \cdots \lambda_n' \omega_n$.

    On the other hand, the branching rule of $\chi(L(m'))$ has only $(\max \{j~|~\lambda_j \neq 0\} + 1)$ terms in the decomposition. When $\lambda_n \neq 0$, we have
    \[\Pi(\chi(L(m'))) = \sum_{s=0}^n \chi^{\sigma} (L(\prod_{i=1}^{n-1} \prod_{k=1}^{\lambda_i} Z_{i,q^{-2a_{i} - (i-1) -2k}} \times \prod_{l=1}^{\lambda_n/2 -1} Z_{n, q^{-2a_{n} - (n-1) - 2l} } \times Z_{n-s,q^{-2a_{n} - (n-1) - 2l-s}})).\]
    When $N = \max \{j~|~\lambda_j \neq 0\} < n$, we have 
    \[\Pi(\chi(L(m'))) = \sum_{s=0}^N \chi^{\sigma} (L(\prod_{i=1}^{N-1} \prod_{k=1}^{\lambda_i} Z_{i,q^{-2a_{i} - (i-1) -2k}} \times \prod_{l=1}^{\lambda_N -1} Z_{N, q^{-2a_{N} - (N-1) - 2l} } \times Z_{N-s,q^{-2a_{N} - (N-1) - 2l-s}})).\]
    
    While each term is a snake module, it is not a minimal affinization in general.
\end{example}

\section{Langlands dual from \texorpdfstring{$A_{2n-1}^{(2)}$}{} to \texorpdfstring{$B_n^{(1)}$}{}}\label{sec:A to B}
Having established the Langlands branching rule from type $B_n^{(1)}$ to type $A_{2n-1}^{(2)}$, we now turn to the question of Langlands duality in the opposite direction, namely from $A_{2n-1}^{(2)}$ to $B_n^{(1)}$.

In \cite[Theorem~6.8]{frenkel2011langlandsfinite}, it is also proved that Kirillov-Reshetikhin modules of twisted type admit Langlands dual representations, using the interpolating $(q,t)$-characters. In this section, we provide an alternative proof of \cite[Theorem~6.8]{frenkel2011langlandsfinite} for the case of type $A_{2n-1}^{(2)}$ to type $B_n^{(1)}$, using the path description. 


As above, we denote by $\Uqghat$ the non-twisted quantum affine algebra of type $B_n^{(1)}$ and $\Uqghatdual$ the twisted quantum affine algebra of type $A_{2n-1}^{(2)}$.

\begin{definition}
    Recall that $P$ is the weight lattice of type $B_n$, $\leftindex^L{P}$ is the dual to the root lattice of type $B_n$. 
    
    Let $\leftindex^L {P}'$ be the sublattice of $\leftindex^L {P}$ defined by $\leftindex^L {P}' = \bigoplus_{i \in I} d_i \check{\omega}_{\bar{i}}$, where $d_i = 2$ for $i \neq n$, and $d_n = 1$.
    
    There is a bijective linear map
    \[\leftindex^L{P}' \to P, \quad d_i \check{\omega}_{\bar{i}} \mapsto \omega_i, \forall i \]
    and we extend it to a surjective map
    \[\leftindex^L {\Pi} :\leftindex^L{P} \to P\]
    such that $\leftindex^L {\Pi}(\lambda) = 0$ if $\lambda \notin \leftindex^L {P}'$.
\end{definition}

We remark that, unlike the weight lattice of type $B_n$, where the root lattice $R \subset P'$ is a sublattice of $P'$, in the dual case, $\leftindex^L{R}$ is not a sublattice of $\leftindex^L{P}'$. Therefore, the definition must be extended to the entire lattice $\leftindex^L{P}$ so that the expression $\leftindex^L{\Pi}(\chi^{\sigma}(V))$ makes sense.

\begin{lemma}\label{lemma: tensor irreducible}
    For $i \in I$ and $T \in \NN^*$, the tensor product of Kirillov-Reshetikhin modules of $\Uqghatdual$
    \[L(Z_{i,a}Z_{i,aq^2}\cdots Z_{i,aq^{2T-2}})\otimes L(Z_{i,-a}Z_{i,-aq^2}\cdots Z_{i,-aq^{2T-2}})\]
    is irreducible.
\end{lemma}
\begin{proof}
    It is well-known that for $1 \le i \le n$ the representation
    \[L(Y_{i,a}Y_{i,aq^2}\cdots Y_{i,aq^{2T-2}})\otimes L(Y_{2n-i,a}Y_{2n-i,aq^2}\cdots Y_{2n-i,aq^{2T-2}})\]
    of the quantum affine algebra of type $A_{2n-1}^{(1)}$ is irreducible. Moreover, it is special in the sense that its $q$-character has a unique dominant monomial. This is due to the fact that Kirillov-Reshetikhin modules are right-minuscule, see, for example, \cite{hernandez2006kirillov}. As a consequence, the tensor product is isomorphic to the irreducible $l$-highest weight module 
    \[V := L(Y_{i,a}Y_{i,aq^2}\cdots Y_{i,aq^{2T-2}}Y_{2n-i,a}Y_{2n-i,aq^2}\cdots Y_{2n-i,aq^{2T-2}}).\]

    For the same reason as in \cite[Theorem~4.15]{hernandez2010kirillov},
    \[\pi(\chi_q(V))= \chi_q^{\sigma}(L(Z_{i,a}Z_{i,aq^2}\cdots Z_{i,aq^{2T-2}}Z_{i,-a}Z_{i,-aq^2}\cdots Z_{i,-aq^{2T-2}})).\]

    The module on the right-hand side is a subquotient of the tensor product 
    \[L(Z_{i,a}Z_{i,aq^2}\cdots Z_{i,aq^{2T-2}})\otimes L(Z_{i,-a}Z_{i,-aq^2}\cdots Z_{i,-aq^{2T-2}}),\]
    and they have the same $q$-character. Therefore, the tensor product is isomorphic to 
    \[L(Z_{i,a}\cdots Z_{i,aq^{2T-2}} Z_{i,-a}\cdots Z_{i,-aq^{2T-2}}),\]
    thus irreducible.
\end{proof}

These modules are called generalized Kirillov-Reshetikhin modules in \cite[Section~6.3]{frenkel2011langlandsfinite}. 

\begin{definition}
    A \textit{generalized Kirillov-Reshetikhin module} of $\Uqghatdual$ is 
    \[W_{T,a}^{(i)} = L(Z_{i,a}\cdots Z_{i,aq^{2T-2}} Z_{i,-a}\cdots Z_{i,-aq^{2T-2}}), \quad \text{$T \in \NN^*$, $a \in \CC^*$}.\]
\end{definition}

By the above lemma, the path description of type $A$ implies that a monomial in the $q$-character $\chi_q(W^{(i)}_{T,a})$ corresponds to a pair of $T$-tuples of NOP.

\begin{proposition}
    Let $\overline{\mathscr{P}}^A_{1}$ be the set of $T$-tuples of NOP associated to the snake 
    $$(i,k)(i,k+2)\cdots(i,k+2T-2).$$ 
    Let $\overline{\mathscr{P}}^A_{2}$ be the set of $T$-tuples of NOP associated to the snake 
    $$(2n-i,k)(2n-i,k+2)\cdots(2n-i,k+2T-2).$$ 
    Then
    \[\chi_q(W^{(i)}_{T,q^k}) = \sum_{(\overline{p}_1,\overline{p}_2) \in \overline{\mathscr{P}}^A_{1} \times \overline{\mathscr{P}}^A_{2}} \mathfrak{m}(\overline{p}_1) \mathfrak{m}(\overline{p}_2).
    \]
\end{proposition}

Our goal is to construct an injective map from $\overline{\mathscr{P}}^B_{(i_t,k_t)}$ associated to the snake 
\[(i_t,k_t)_{1 \leq t \leq T} = \big( (i,2k)(i,2k+4)\cdots(i,2k+4T-4) \big)\] to $\overline{\mathscr{P}}^A_{1} \times \overline{\mathscr{P}}^A_{2}$.

As above, we begin with the construction for a single path, then we pass to a $T$-tuple of NOP.

\begin{definition}
    Let $p \in \mathscr{P}^B_{i,2k}$, $(i,2k) \in \mathcal{X}^B$. We define a pair of paths in $\mathscr{P}^A_{i,k} \times \mathscr{P}^A_{2n-i,k}$ in the following way.

    Write 
        \begin{align*}
           && p = \big( (0,y_0), &(2,y_1), \dots, (2n-4, y_{n-2}), (2n-2,y_{n-1}), (2n-1,y_n), \\
           && &(2n-1,z_n), (2n,z_{n-1}),\dots,(4n-4,z_1),(4n-2,z_0) \big).
        \end{align*}

    Let $j_0 = \max\{j| 1 \le j \le n, y_{j-1} \le z_{j-1}+2\}$. Such a $j_0$ exists, because $y_0 = 2i+2k$, $z_0 = 4n-2i+2k-2$, thus $y_0 \le z_0 +2$.

    Define $G(p)_L$ to be the path in $\mathscr{P}^A_{i,k}$:
    \[G(p)_L = \big( (0,x_0),(1,x_1),\dots,(2n,x_{2n}) \big)\]
    such that 
    \[x_j = 
    \begin{cases}
        y_j/2, &\text{if}~0 \le j \le n-1,\\
        \lfloor (y_n +1)/2 \rfloor, &\text{if}~j=n, \\
        y_{2n-j}/2, &\text{if}~n+1 \le j \le 2n-j_0,\\
        (z_{2n-j}+2)/2, &\text{if}~2n-j_0+1 \le j \le 2n.
    \end{cases}
    \]

    Similarly, define $G(p)_R$ to be the path in $\mathscr{P}^A_{2n-i,k}$:
    \[G(p)_R = \big( (0,x_0'),(1,x_1'),\dots,(2n,x_{2n}') \big)\]
    such that 
    \[x_j' = 
    \begin{cases}
        (z_j+2)/2, &\text{if}~0 \le j \le n-1,\\
        \lfloor (z_n +3)/2 \rfloor, &\text{if}~j=n, \\
        (z_{2n-j}+2)/2, &\text{if}~n+1 \le j \le 2n-j_0,\\
        y_{2n-j}/2, &\text{if}~2n-j_0+1 \le j \le 2n.
    \end{cases}
    \]

    Denote by $G(p)$ the pair of paths $(G(p)_L,G(p)_R)$. 
\end{definition}

\begin{lemma}
    We show that this defines a map 
    \[G: \mathscr{P}^B_{i,2k} \to \mathscr{P}^A_{i,k} \times \mathscr{P}^A_{2n-i,k}.\]
\end{lemma}
\begin{proof}
    Let $j_0 = \max\{j| 1 \le j \le n, y_{j-1} \le z_{j-1}+2\}$. We claim that $y_{j_0-1} = z_{j_0-1} +2$.
    
    If $j_0 \neq n$, then $\lvert y_{j_0} - y_{j_0-1} \rvert = 2$ and $\lvert z_{j_0} -z_{j_0-1} \rvert =2$. Thus $y_{j_0} > z_{j_0}+2$ implies that $0 \leq y_{j_0-1} - z_{j_0-1}-2 < 4$. Notice that $\forall j \neq n$, $\lvert y_j - z_j \rvert \equiv 2 \pmod 4$, thus the difference $y_{j_0-1} - z_{j_0-1}-2$ must be $0$.

    If $j_0 = n$, then $\lvert y_{n-1} - y_n \rvert = 1+\epsilon$ and $\lvert z_{n-1} - z_n \rvert = 1+\epsilon$. Thus $y_n > z_n$ implies that $0 \leq y_{n-1} - z_{n-1}-2 < 2\epsilon$. The difference is an integer, thus must be $0$.
    
    Therefore, the point $(2n-j_0+1,(z_{j_0-1}+2)/2)$ in $G(p)_L$ and $G(p)_R$ can also be written as $(2n-j_0+1,y_{j_0-1}/2)$.

    Then it is easy to check that $G(p)_L$ and $G(p)_R$ are type $A_{2n-1}$ paths in $\mathscr{P}^A_{i,k}$ and $\mathscr{P}^A_{2n-i,k}$, respectively.
\end{proof}

The following proposition is a direct consequence of the definition.
\begin{proposition}
    The map $G: \mathscr{P}^B_{i,2k} \to \mathscr{P}^A_{i,k} \times \mathscr{P}^A_{2n-i,k}$ is injective.
\end{proposition}
\begin{proof}
    For two paths $p$ and $p'$ in $\mathscr{P}^B_{i,2k}$, the equality $G(p)_L = G(p')_L$ implies that $p$ and $p'$ have the same left branch, and $G(p)_R = G(p')_R$ implies that they have the same right branch. Therefore, the map $G$ is injective.
\end{proof}

The following theorem is the analogue of Theorem~\ref{thm: map between paths preserves the character}.

\begin{theorem}
    The map $G$ preserves the weight of paths in the sense that for any path $p \in \mathscr{P}^B_{i,2k}$, 
    \[\mathfrak{m}'(p) = \leftindex^L{\Pi} \circ \varpi(\mathfrak{m}'(G(p)_L) + \mathfrak{m}'(G(p)_R)).\]
\end{theorem}
\begin{proof}
    We consider the set of points with multiplicity
    \[ G(p)_L \sqcup G(p)_R.\]
    This set consists of the disjoint union of two paths of type $A_{2n-1}$, and we call it a double path in the following. 
    
    We remark that there may be several different ways to decompose a double path as a union of two paths of type $A_{2n-1}$. However, given a double path of type $A_{2n-1}$, if we decompose it as $p_1 \sqcup p_2$, then we have 
    that the multiplicity of $\omega_j$ ($\forall j$) in $\mathfrak{m}'(p_1) + \mathfrak{m}'(p_2)$ equals to
    \[\frac{1}{2}(a_{j-1} - 2a_j + a_{j+1} + b_{j-1} - 2b_j + b_{j+1}),\]
    as we have seen in the proof of Proposition~\ref{prop: bijection preseves half-character}, where $(j-1,a_{j-1}),(j,a_j), (j+1,a_{j+1}),(j-1,b_{j-1}),(j,b_j),(j+1,b_{j+1}) \in p_1 \sqcup p_2$ with multiplicity.

    We note that this expression is independent of the decomposition $p_1 \sqcup p_2$. Therefore, if $p_1' \sqcup p_2' = p_1 \sqcup p_2$ as set of points with multiplicity, then $\mathfrak{m}'(p_1) + \mathfrak{m}'(p_2) = \mathfrak{m}'(p_1') + \mathfrak{m}'(p_2')$.

    In our case, we can also decompose 
    \[G(p)_L \sqcup G(p)_R = p_y \sqcup p_z,\]
    where 
    \[p_y = \big( (0,\frac{y_0}{2}),(1,\frac{y_1}{2}),\dots,(n-1,\frac{y_{n-1}}{2}),(n,\lfloor \frac{y_n+1}{2} \rfloor),(n+1,\frac{y_{n-1}}{2}),\dots,(2n-1,\frac{y_{1}}{2}),(2n,\frac{y_{0}}{2}) \big), \]
    \[p_z = \big( (0,\frac{z_0+2}{2}),\dots,(n-1,\frac{z_{n-1}+2}{2}),(n,\lfloor \frac{z_n+3}{2} \rfloor),(n+1,\frac{z_{n-1}+2}{2}),\dots,(2n,\frac{z_{0}+2}{2}) \big). \]

    Therefore,
    \[\leftindex^L{\Pi} \circ \varpi(\mathfrak{m}'(G(p)_L) + \mathfrak{m}'(G(p)_R)) = \leftindex^L{\Pi} \circ \varpi(\mathfrak{m}'(p_y) + \mathfrak{m}'(p_z)).\]

    Since the Langlands dual map $\leftindex^L{\Pi}$ acts by 
    \begin{align*}
        &\omega_{\bar{i}} \mapsto \omega_{i}/2, \quad i \neq n,\\
        &\omega_{\bar{n}} \mapsto \omega_{n},
    \end{align*}
    we have 
    \[\leftindex^L{\Pi} \circ \varpi(\mathfrak{m}'(p_y) + \mathfrak{m}'(p_z))  = \mathfrak{m}'(p).\]
\end{proof}

For a $T$-tuple of NOP $\overline{p} = (p_1,\dots,p_{T}) \in \overline{\mathscr{P}}^B_{(i_t,k_t)}$ associated to the snake 
\[(i_t,k_t)_{1 \leq t \leq T} = \big( (i,2k)(i,2k+4)\cdots(i,2k+4T-4) \big),\]
we define $G(\overline{p})$ to be the pair of $T$-tuples of paths \[\big(G(p_1)_L,G(p_2)_L,\dots,G(p_{T})_L \big)\] and \[\big( G(p_1)_R,G(p_2)_R,\dots,G(p_{T})_R \big).\]

\begin{theorem}
    $G$ maps a $T$-tuple of NOP to a pair of $T$-tuples of NOP    
    \[G: \overline{\mathscr{P}}^B_{(i_t,k_t)} \to \overline{\mathscr{P}}^A_{1} \times \overline{\mathscr{P}}^A_{2}.\]
\end{theorem}
\begin{proof}
    We need to show that if $p_1 \succ p_2$, then $G(p_1)_L \succ G(p_2)_L$ and $G(p_2)_R \succ G(p_2)_R$. For points whose first coordinate is not $n$, the inequality is obvious. For points with first coordinate $n$, the same argument in the proof of Proposition~\ref{prop: level 0 bijection} also applies here word by word.
\end{proof}

\begin{corollary}
    Let $W_{T,q^k}^{(i)}$ be a generalized Kirillov-Reshetikhin module of $\Uqghatdual$. Then the usual character of the Kirillov-Reshetikhin module $L(Y_{i,q^k}\cdots Y_{i,q^{k+2T-2}})$ of $\Uqghat$ is dominated by that of $W_{T,q^k}^{(i)}$:
    \[\chi(L(Y_{i,q^k}\cdots Y_{i,q^{k+2T-2}})) \preceq \leftindex^L{\Pi} (\chi^{\sigma}(W_{T,q^k}^{(i)})). \]
\end{corollary}
\begin{proof}
    The proof is identical to that of Theorem~\ref{thm: langlands dual rep A to B}.
\end{proof}

\begin{remark}
    Following the same proof, it can be shown that for any shortened snake $(i_t,k_t)_{1 \le t \le T}$ of type $B_{n}$, 
    \[\chi(L(m)) \preceq \leftindex^L{\Pi} (\chi^{\sigma}(L({}^L{m}) \otimes L({}^L{m}'))), \]
    where $m$ is the monomial associated to the shortened snake $(i_t,k_t)_{1 \le t \le T}$ as defined in \eqref{eqn: snake weight type B}, ${}^L{m}$ is the monomial in \eqref{eqn: snake weight Lm}, and ${}^L{m}'$ is the monomial obtained from ${}^L{m}$ by replacing the spectral parameters $q^k$ with $-q^k$, $\forall k \in \ZZ$.

    However, in contrast to Kirillov-Reshetikhin modules, the tensor product of snake modules $L({}^L{m}) \otimes L({}^L{m}')$ is not irreducible in general.
\end{remark}

Concerning the positivity conjecture in the opposite direction, it is conjectured in \cite{frenkel2011langlandsfinite} that for an irreducible finite-dimensional representation $V$ of $\Uqghatdual$ of type $A_{2n-1}^{(2)}$, the image $\leftindex^L{\Pi}(\chi^{\sigma}(V))$ can be decomposed into a sum of characters of finite-dimensional representations of $\Uqghat$. 

However, we remark that in contrast to the map $\Pi: P' \to {}^L{P}$, which is injective, in the opposite direction, the expression $\leftindex^L{\Pi}(\chi^{\sigma}(V))$ contains fewer terms than $\chi^{\sigma}(V)$, because the root lattice $\leftindex^L{R}$ is not a sublattice of $\leftindex^L {P}'$ when $n \ge 3$.

\end{document}